\documentclass[10pt, a4paper]{article} 

\usepackage[utf8]{inputenc} 

\usepackage{graphicx} 
\usepackage[english]{babel}
\usepackage{xcolor}

\usepackage{amsfonts,amsmath,amssymb,bbm,amsthm,amsbsy}
\usepackage{tikz}
\usetikzlibrary{calc}
\usetikzlibrary{cd}

\usepackage[pagebackref,%colorlinks,linkcolor=black,citecolor=black%
    ,pdfborder={0 0 0}
  ,urlcolor=black,a4paper,hypertexnames=false]{hyperref}

\usepackage{color}
\usepackage{pdfcolmk}

\makeatletter
\def\blfootnote{\xdef\@thefnmark{}\@footnotetext}
\makeatother
\date{\today%
    \protect\blfootnote{\copyright{\ F.~Hofmann, C.~L\"oh 2026}. 
      This work was partially supported by the CRC~1085 \emph{Higher Invariants} and 
      the CRC~1785 \emph{Generalised Motivic Methods in Geometry}
      (Universit\"at Regensburg, funded by the DFG).
      Parts of this article are based on questions that arose
      during the Oberwolfach Seminar \emph{Metric Topology of
        Aspherical Spaces} (MFO, 2024).
    \\
    MSC~2020 classification:
    55U15, %Chain complexes in algebraic topology
    55N10, %Singular homology and cohomology theory
    55U10, %Simplicial sets and complexes in algebraic topology
    55U25%  %Homology of a product, Künneth formula
    }}

\usepackage[nouppercase]{scrlayer-scrpage}

\automark[subsection]{section}
\usepackage{ifthen}
\deftripstyle{myheadings}%{\headmark}{}{\pagemark}%
             {\ifthenelse{\isodd{\value{page}}}{\leftmark}{\pagemark}}
             {}
             {\ifthenelse{\isodd{\value{page}}}{\pagemark}{\rightmark}}% 
             {}{}{}% 
\def\args{\;\cdot\;}

\newtheorem{thm}{Theorem}[section]
\newtheorem{lemma}[thm]{Lemma}
\newtheorem{prop}[thm]{Proposition}
\newtheorem{cor}[thm]{Corollary}

\newtheorem*{claim*}{Claim}

\theoremstyle{definition}
\newtheorem{defi}[thm]{Definition}
\newtheorem{example}[thm]{Example}
\newtheorem{setup}[thm]{Setup}
\newtheorem{question}[thm]{Question}

\newtheorem{notation}[thm]{Notation}
\newtheorem{rem}[thm]{Remark}

\newcommand{\R}{\mathbb{R}}
\newcommand{\F}{\mathbb{F}}
\newcommand{\Q}{\mathbb{Q}}
\newcommand{\Z}{\mathbb{Z}}
\newcommand{\N}{\mathbb{N}}

\DeclareMathOperator{\MatM}{M}
\def\mat#1#2#3{\MatM_{#1 \times #2}(#3)}

\DeclareMathOperator{\cell}{cell}
\def\ccell#1{%
  C_{#1}^{\cell}}
\DeclareMathOperator{\GL}{GL}

\newcommand{\seminorm}[3]{\ensuremath{\left\| {#1}\right\|}_{#2}^{#3}} % norm with 1=input 2=subscript 3=supscript
\newcommand{\normsum}[1]{\ensuremath{\seminorm{{#1}}{\otimes}{}}} % norm on left side of Künneth, direct sum of tensor of homologies
\newcommand{\normtensorcplx}[2]{\ensuremath{\left|{{#1}}\right|^{{#2}\otimes D}}} % norm on C\otimes_\R D
\newcommand{\normhom}[2]{\ensuremath{\left\| {#1}\right\|^{{#2}\otimes D}}} % norm on homology of C$\otimes_\R$ D, right side of Künneth
\newcommand{\bignormhom}[2]{\ensuremath{\bigl\| {#1}\bigr\|^{{#2}\otimes D}}} % big norm on homology of C$\otimes_\R$ D, right side of Künneth

\DeclareMathOperator{\id}{id}
\DeclareMathOperator{\im}{im}
\DeclareMathOperator{\Hom}{Hom}

\DeclareMathOperator{\rk}{rk}
\DeclareMathOperator{\ev}{ev} % evaluation map
\DeclareMathOperator{\Null}{N} % Nullspace of semi-normed space

\DeclareMathOperator{\mCh}{{\sf mCh}}
\DeclareMathOperator{\finop}{{\sf fin}}
\def\mChfin#1{%
  \mCh^{\finop}_k}

\def\qand{%
  \quad\text{and}\quad}
\def\fa#1{%
  \forall_{#1}\;\;\;}
\def\exi#1{%
  \exists_{#1}\;\;\;}

\def\epsilon{\varepsilon}

\title{Quantitative chain extensions\\ and integral cohomological fillings}
\author{Franziska Hofmann and Clara L\"oh}

\begin{document}

\maketitle

\begin{abstract}
  We develop two methods in quantitative homological 
  algebra: A declarative approach to
  integral cohomological filling problems and 
  a quantitatively controlled replacement of
  normed chain complexes by their homology complex.
  Both methods are connected through so-called
  chain extensions. As applications, we obtain
  transparent proofs for uniform integral simplicial
  cohomological fillings for subdivisions of simplicial
  complexes, for uniform integral cohomological
  fillings of certain finite coverings, and for lower
  bounds of the $\ell^1$-semi-norm on product spaces.
\end{abstract}

%%%%%%%%%%%%%%%%%%%%%%%%%%%%%%%%%%%%%%%%%%%%%%%%%%%%%%%%%%%%%%%%%
\section{Introduction}

Quantitative homological algebra enriches chain
complexes or homology with norms. 
In geometric topology, quantitative homological
algebra is useful in many contexts; examples relevant for
this paper are geometric rigidity via simplicial volume or bounded cohomology~\cite{vbc},
quantitative cobordism via integral cohomological fillings~\cite{CDMW},
and waist inequalities via integral cohomological fillings~\cite{bader_sauer}.

%%%%%%%%%%%%%%%%%%%%%%%%%%%%%%%%
\subsection*{Fundamental concepts}

We develop two methods in quantitative homological
algebra: 
\begin{itemize}
\item
  A declarative approach to quantitative cohomological lifting over~$\R/\Z$ 
  and thus to translating cohomological
  filling problems over~$\Z$ into cohomological filling
  problems over~$\R$.
\item
  Quantitatively controlled replacements 
  of normed $\R$-chain complexes by their homology complex,
  i.e., by chain complexes with trivial boundary operator.
\end{itemize}
Both methods are connected through the simple algebraic
concept of \emph{chain extensions} (Section~\ref{sec:chainext}). 
A chain extension in degree~$k$ is a map from the $k$-th chain module 
to the $k$-th homology extending the canonical projection from the
$k$-cycles to homology. In our quantitative context, we will
consider different norms of chain extensions. 

In the first method, we introduce size invariants on marked free chain
complexes measuring the homological complexity of the boundary
operator, while ignoring the size of the homology
(Section~\ref{sec:size}). Size invariants satisfy basic 
normalisation and monotonicity properties.

A concrete example of such a size invariant is the cohomological
lifting volume. By definition, it is the optimal quantity for lifts
under cohomological change-of-coefficients maps of the
form~$H^*(C_*;\R) \to H^*(C_*;\R/\Z)$. The monotonicity
of the cohomological lifting volume exhibits (norms of)
chain extensions as a pragmatic tool for upper bounds. 
The cohomological lifting volume controls the passage from
cohomological filling problems over~$\Z$ to cohomological filling
problems over~$\R$, which are usually more accessible.

A second example of a size invariant is the so-called gnarledness of
chain complexes, which is the optimal quantity for the size of chain
extensions with respect to cointegral bases in homology. By design,
gnarledness is the universal size invariant. This notion of
gnarledness is an algebraic version of the gnarledness of simplicial
complexes by Chambers, Dotterrer, Manin, and
Weinberger~\cite{CDMW}. In contrast with the geometric setting,
algebraic gnarledness enjoys a straightforward monotonicity.

In the second method, we use chain extensions to replace chain
complexes by chain complexes with trivial boundary operator.
Gnarledness measures the size of such replacements with
respect to cointegral bases in homology. Measuring the size of
chain extensions of normed chain
complexes instead with respect to the induced semi-norm
on homology leads to lower bounds for the induced semi-norm on
the homology of tensor product complexes.

%%%%%%%%%%%%%%%%%%%%%%%%%%%%%%%%
\subsection*{Applications}

We obtain transparent proofs
for uniform integral simplicial cohomological fillings
for subdivisions of simplicial complexes and for uniform
integral cellular cohomological fillings of certain
finite coverings. Furthermore, we obtain lower bounds
for the $\ell^1$-semi-norm on product spaces. We now
outline these applications in more detail.

For simplicial subdivisions, we have the following version of a result
of Chambers, Dotterrer, Manin, and
Weinberger~\cite[Lem\-ma~3.9]{CDMW}, which they use to quantitatively
control null-cobordisms. The original proof is based on geometric
gnarledness and the explicit construction of controlled spanning trees
in simplicial subdivisions. Our algebraic approach via the
cohomological lifting volume bypasses such geometric constructions and
is thus robust enough to work in a uniform way for all simplicial
subdivisions.

\begin{cor}[integral cohomological fillings for subdivisions;
  Corollary~\ref{cor:fillsubdiv}]
  Let $X$ be a finite simplicial complex, let $k \in \N$,
  and let $\F$ be a subfield of~$\R$. Then there exists
  a constant~$c \in \R_{>0}$ such that:
  If $X'$ is a simplicial subdivision of~$X$ and
  $f \in C_\Delta^k(X';\F)$ is a simplicial cochain such that $\delta f$ is a coboundary
  in~$C_\Delta^*(X';\Z)$, then there exists a cochain~$\widetilde f \in C_\Delta^k(X';\Z)$
  with
  \[ \delta \widetilde f = \delta f
  \qand
  |\widetilde f|_\infty \leq |f|_\infty +c.
  \]
\end{cor}

For this application, in addition to the basic notions, only
Sections~\ref{sec:size}, \ref{subsec:defliftingvol},
\ref{subsec:liftfill}, and Section~\ref{sec:liftsc} are needed. To put
our approach into context, we develop size invariants (Section~\ref{sec:size})
and algebraic gnarledness of chain complexes (Section~\ref{sec:gnarl}).
This includes the comparison between algebraic gnarledness and the cohomological lifting volume
(Section~\ref{subsec:liftgnarl}), the comparison between algebraic
gnarledness and geometric gnarledness (Section~\ref{subsec:gnarlgeo}),
a rephrasing of simplicial volume in terms of algebraic
gnarledness of the singular chain complex
(Section~\ref{subsec:gnarlsv}), and the universality of gnarledness
among size invariants (Section~\ref{subsec:gnarluniversal}). 

As a second application of the cohomological lifting volume, we
obtain a straightforward proof of the relation between
integral and real expansion constants of the cellular boundary
operator in certain finite coverings:

\begin{cor}[Corollary~\ref{cor:fincov}, Remark~\ref{rem:expansionindep}]
  Let $X$ be a finite CW-complex. Then there is a constant~$c \in \R_{>0}$
  such that: For every finite sheeted covering map~$p \colon \overline X \to X$
  (where $\overline X$ is equipped with
  the CW-structure induced by~$p$) and every~$k \in \N$  
  with~$\dim_\R H_k(\overline X;\R) = \dim_\R H_k(X;\R)$, we have
  \[ \Xi_k^\Z\bigl(\ccell*(\overline X) \bigr)
    \leq \Xi_k^\R\bigl(\ccell*(\overline X)\bigr)
    + c
  \]
\end{cor}

In particular, this applies in degree~$0$ and~$1$ to finite
CW-complexes with Kazhdan fundamental group (Example~\ref{exa:kazhdan}),
as used by Bader and Sauer in their proofs of uniform waist
inequalities~\cite{bader_sauer}.  For this application, in addition to
the basic notions, only Sections~\ref{sec:size},
\ref{subsec:defliftingvol}, \ref{subsec:liftfill},
and~\ref{subsec:fincov} are needed.

Finally, using the existence of bounded chain extensions with
respect to the induced semi-norm on homology (Section~\ref{subsec:snboundedchainext}),
we obtain the following result for products. It should be noted
that the calculation of the $\ell^1$-semi-norm on products is
notoriously difficult and only few exact calculations are known~\cite{bucher,heuer_loeh_spectrum}.

\begin{cor}[Corollary~\ref{cor:Kuennethtopo}]
  Let $X$,~$Y$ be topological spaces such that $X$ has finite
  dimensional reduced homology in each degree (e.g., a closed manifold
  or a CW-complex of finite type). Then, for each $n\in\N$, the Künneth isomorphism
  \begin{align*}
    \bigoplus_{p\in \N} H_p(X;\R) \otimes_\R H_{n-p}(Y;\R)
    & \to H_n(X\times Y;\R)
    \\
    \alpha \otimes \beta
    & \mapsto \alpha \times \beta
  \end{align*}
  is bilipschitz.
\end{cor}

In particular, this also leads to an inheritance result
for the behaviour of product maps with respect to the
$\ell^1$-semi-norm (Corollary~\ref{cor:productofbilip}).
For these applications, in addition to basic notions,
only Section~\ref{sec:normedkuenneth} is needed.

%%%%%%%%%%%%%%%%%%%%%%%%%%%%%%%%%%%%%%%%%%%%%%%%%%%%%%%
% toc
\tableofcontents

%%%%%%%%%%%%%%%%%%%%%%%%%%%%%%%%%%%%%%%%%%%%%%%%%%%%%%%%%%%%%%%%%
\section{Chain extensions}\label{sec:chainext}

Chain extensions are maps from a chain complex to its homology that
extend the canonical projection from the space of cycles to the
homology. 

\begin{defi}[chain extension]
  Let $C_*$ be a chain complex over a ring~$R$ and
  $k \in \N$. A \emph{chain extension of~$C_*$ in degree~$k$}
  is an $R$-linear map~$s \colon C_k \to H_k(C_*)$
  with the following property:
  For all cycles~$c \in C_k$, we have 
  \[ s(c) = [c].
  \]
  In other words, $s$ makes the following diagram commutative,
  where the left arrow is the inclusion and the right 
  arrow is the canonical projection:
  \[
  \begin{tikzcd}
    C_k 
    \ar[rr,dashed,"s",description]
    && H_k(C_*)
    \\
    & \ker \partial_k
    \ar{ul}{}
    \ar{ur}{}
  \end{tikzcd}
  \]
\end{defi}

Clearly not every chain complex admits a chain extension.  Chain
complexes over fields always admit chain extensions.  If $C_*$ is a
chain complex and $(s_k)_k$ is a sequence of chain extensions in all
degrees~$k$, then $s_* \colon C_* \to H_*(C_*)$ is a homology
isomorphism, where we view~$H_*(C_*)$ as chain complex with trivial
boundary operator: under the canonical isomorphism~$H_*(H_*(C_*))
\cong H_*(C_*)$, the induced map~$H_*(s_*)$ corresponds
to~$\id_{H_*(C_*)}$. If the chain complex~$C_*$ is free in every
degree (and bounded below), then $s_*$ is a chain homotopy
equivalence~\cite[Chapter~1, Theorem~8.4]{brown_cohomofgroups}.
Thus, in this situation, chain extensions are
a measure gauging ``how far'' a chain complex is from having a trivial
boundary operator.

We will consider two different ways to quantify the size of chain
extensions in the context of marked free chain complexes: Using
$\ell^1$-norms on homology with respect to cointegral bases
(Sections~\ref{sec:size}--\ref{sec:gnarl}) and using the induced
semi-norm on homology (Section~\ref{sec:normedkuenneth}).

In the following, we will mainly be interested in chain complexes
obtained by base change from chain complexes over~$\Z$.

\begin{notation}[(co)homology with coefficients]
  If $C_*$ is a $\Z$-chain complex, $\F$ is a commutative ring,
  and $k \in \N$, we abbreviate
  \begin{align*}
    H_k(C_*;\F)
    & := H_k(C_* \otimes_\Z \F)
    \\
    H^k(C_*;\F)
    & := H_k\bigl(\Hom_\Z(C_*,\F)\bigr).
  \end{align*}
  If $f_*
  \colon C_* \to D_*$ is a $\Z$-chain map, then we write
  \begin{align*}
    H_k(f_*;\F)
    & := H_k(f_* \otimes_\Z \id_\F) \colon H_k(C_*;\F) \to H_k(D_*;\F)
    \\
    H^k(f_*;\F)
    & := H^k\bigl(\Hom_\Z(f_*,\id_\F)\bigr)
    \colon H^k(D_*;\F) \to H^k(C_*;\F).
  \end{align*}
  Furthermore, we write~$H_k(C_*)_\F$ for the integral part
  of~$H_k(C_*;\F)$, i.e., for the image of~$H_k(C_*;\Z)$ under the
  coefficients change~$H_k(C_*;\Z) \to H_k(C_*;\F)$.
  If $\F$ is a field, an $\F$-basis~$B$ of~$H_k(C_*;\F)$ is \emph{cointegral}
  if~$H_k(C_*)_\F \subset \Z B$.
\end{notation}

\begin{defi}[chain extension, with coefficients]
  Let $C_*$ be a $\Z$-chain complex,
  let $k \in \N$, and let $\F$ be a commutative ring.
  A \emph{chain extension 
    of~$C_*$ in degree~$k$ over~$\F$} is a chain
  extension~$s \colon C_k \otimes_\Z \F \to H_k(C_*;\F)$ of~$C_*\otimes_\Z \F$
  in degree~$k$ of the $\F$-chain complex~$C_* \otimes_\Z \F$. 
\end{defi}

\begin{lemma}[transport of chain extensions]\label{lem:chainextheq}
  Let $C_*$, $D_*$ be $\Z$-chain complexes, let $\F$ be a
  commutative ring, let $k \in \N$, let $s$ be a chain extension of~$D_*$
  in degree~$k$ over~$\F$, and
  let $f_* \colon C_* \otimes_\Z \F \to D_* \otimes_\Z \F$
  be an $\F$-chain map that induces an $\F$-isomorphism~$H_k(f_*)
  \colon H_*(C_*;\F) \to H_*(D_*;\F)$. Then 
  \[ H_k(f_*)^{-1} \circ s \circ f_k 
  \colon C_k \otimes_\Z \F \to H_k(C_*;\F)
  \]
  is a chain extension of~$C_*$ in degree~$k$ over~$\F$.
\end{lemma}
\begin{proof}
  This is a straightforward calculation:
  Let $c \in C_k \otimes_\Z \F$ be a cycle. Then
  $f_k(c) \in D_k \otimes_\Z \F$
  is a cycle and as $s$ is a chain extension
  of~$D_*$ in degree~$k$ over~$\F$ we obtain (in~$H_k(C_*;\F)$)
  \begin{align*}
    [c]
    & = H_k(f_*)^{-1} \circ H_k(f) [c]
    = H_k(f_*;\F)^{-1} \circ s \bigl(f_k(c)\bigr)
    ,
  \end{align*}
  as desired.
\end{proof}

The hypotheses on~$f_*$ in Lemma~\ref{lem:chainextheq} are, for
instance, satisfied if $f_* = \widetilde f_* \otimes_\Z \id_\F$,
where $\widetilde f_* \colon C_* \to D_*$ is a $\Z$-chain
homotopy equivalence: Indeed, then~$f_* = \widetilde f_*\otimes_\Z
\id_\F \colon C_* \otimes_\Z \F \to D_* \otimes_\Z \F$ is an
$\F$-chain homotopy equivalence and so $f_*$ indeed induces an
$\F$-isomorphism~$H_k(f_*;\F) \colon H_k(C_*;\F) \to H_k(D_*;\F)$.

%%%%%%%%%%%%%%%%%%%%%%%%%%%%%%%%%%%%%%%%%%%%%%%%%%%%%%%%%%%%%%%%%
\section{Basic notions in normed homological algebra}\label{sec:basicnorms}

We recall basic notions and properties of norms, semi-norms,
normed chain complexes, and the induced semi-norm on cohomology.

\begin{setup}
  In this section, $\F$ will be a subfield of the reals with
  the standard absolute value.
\end{setup}

\begin{defi}[norm]
  A \emph{norm on an $\F$-vector space~$X$} is a
  map~$X\to \R_{\geq 0}$ that fulfils the triangle inequality, 
  is $\F$-homogeneous, and positive definite.
\end{defi}

\begin{defi}[semi-norm]
  Let $X$ be an $\F$-vector space. A \emph{semi-norm on~$X$} is
  a map~$\|\cdot \| \colon X\to \R_{\geq 0}$ that fulfils the
  triangle inequality and is $\F$-homogeneous.
\end{defi}

More generally, we will also consider norms and semi-norms
with values in the extended range~$[0,\infty]$ (requiring
$0$ being mapped to~$0$).

\begin{defi}[bounded map]
  If $f \colon X \to Y$ is an $\F$-linear map, then
  we write
  \[\|f\| := \inf \bigl\{ c\in \R_{\geq 0}
  \bigm| \fa{x\in X} \bigl\|f(x)\bigr\| \leq c \cdot \|x\|
  \bigr\}
  \in [0,\infty].
  \] 
  We call $f$ \emph{bounded} if~$\|f\| < \infty$.
\end{defi}
  
\emph{Normed $\F$-chain complexes} are $\F$-chain complexes, where all
chain modules are equipped with norms (or, more generally, semi-norms)
and such that the boundary operators are bounded in each degree with
respect to the norms on the chain modules. The norm on a normed chain
complex induces a semi-norm on the corresponding homology, the
\emph{induced semi-norm}. 
We will mainly be interested in normed chain complexes arising in the
following way:

\begin{defi}[marked free chain complex]
  Let $R$ be a ring.  A \emph{marked free $R$-chain complex} is an
  $R$-chain complex~$C_*$ that is free in each degree and equipped
  with a fixed (``marked'') basis in every degree.
\end{defi}
  
If $C_*$ is a marked free $\Z$-chain complex, we view $C_* \otimes_\Z
\F$ as a marked free $\F$-chain complex through the induced $\F$-bases
and consider the corresponding $\ell^1$-norm~$|\cdot|_1$ on~$C_* \otimes_\Z \F$.

In particular, singular chain complexes, simplicial chain complexes, cellular
chain complexes, and bar complexes carry corresponding $\ell^1$-norms.
The boundary operators are bounded in the singular, simplicial, and bar cases
or if the underlying chain module is of finite rank.

\begin{notation}
  If $C_*$ is a marked free $\Z$-chain complex with marked basis~$\Sigma_k$ in
  degree~$k \in \N$, then the cochain complex~$\Hom_\Z (C_*,\F)$ 
  carries the dual norm
  \begin{align*}
    |\cdot|_\infty \colon \Hom_\Z(C_k,\F)
    & \rightarrow [0,\infty]
    \\
    f & \mapsto \sup_{b \in \Sigma_k} \bigl|f (b) \bigr|.
  \end{align*}
  The induced semi-norm on~$H^*(C_*;\F)$ is denoted by~$\|\cdot\|_\infty$.
  The coboundary operator of~$\Hom_\Z(C_*,\F)$ is bounded in degree~$k$
  if the boundary operator of~$C_* \otimes_\Z \F$ is bounded in degree~$k+1$.
\end{notation}

In particular, this construction applies to singular cochains,
simplicial cochains, cellular cochains, and bar cochains.

\begin{rem}\label{rem:passboundedtohomology}
  Let $C_*$, $D_*$ be marked free $\Z$-chain complexes and
  $f_* \colon C_* \to D_*$ a $\Z$-chain map that is bounded
  in degree~$k \in \N$. Then, 
  \[ \bigl\| H_k(f_*;\F) \bigr\| \leq \|f_k\|
  \qand
  \bigl\| H^k(f_*;\F) \bigr\| \leq \|f_k\|,
  \]
  where the operator norms are taken 
  with respect to the induced $\ell^1$-semi-norm and $\ell^\infty$-semi-norm
  on (co)homology, respectively.
\end{rem}

%%%%%%%%%%%%%%%%%%%%%%%%%%%%%%%%%%%%%%%%%%%%%%%%%%%%%%%%%%%%%%%%
\section{Size invariants on marked free chain complexes}\label{sec:size}

We consider a simple type of size invariants on marked free $\Z$-chain
complexes, satisfying basic normalisation and monotonicity
properties.  Here, ``size'' does not refer to the size in terms of
the size of homology, but to the homological complexity of the boundary
operator. This is reflected in the normalisation condition. The key
property for calculations or meaningful estimates is monotonicity,
which leads to the chain extension estimate and to finiteness results.

Examples of such size invariants are the cohomological lifting volume
(Section~\ref{sec:liftingvol}) and gnarledness of chain complexes
(Section~\ref{sec:gnarl}). The former is optimal for cohomological
lifting and related to integral cohomological filling problems; the
latter is, by design, universal among size invariants
(Section~\ref{subsec:gnarluniversal}). The sandwich estimate between
the cohomological lifting volume and gnarledness
(Proposition~\ref{prop:liftgnarl}) expresses that both are closely
related and the cohomological lifting volume is almost universal.

\begin{notation}
  We write $\mCh$ for the class of marked free $\Z$-chain complexes. 
  We equip chain complexes in~$\mCh$ (and their base-changed versions)
  in each degree with the $\ell^1$-norm~$|\cdot|_1$ given by the marked
  basis.
\end{notation}

\begin{defi}[size invariant]\label{def:size}
  Let $\F$ be a subfield of~$\R$ and let $k \in \N$.
  A \emph{size invariant over~$\F$ in degree~$k$} is a
  map~$I \colon \mCh \to [0,\infty]$ with the following properties:
  \begin{itemize}
  \item
    \emph{Normalisation.}
    If $C_*$ is in~$\mCh$ with~$\partial_k = 0$ and $\partial_{k+1} = 0$,
    then~$I(C_*) \leq 1$.
  \item
    \emph{Monotonicity.}
    If $C_*$, $D_*$ are in~$\mCh$ and $f_* \colon C_* \otimes_\Z \F
    \to D_* \otimes_\Z \F$ is an $\F$-chain map inducing an 
    $\F$-isomorphism~$H_k(f_*) \colon H_k(C_*;\F) \to H_k(D_*;\F)$
    satisfying the integrality condition~$H_k(f_*) (H_k(C_*)_\F) \subset
    H_k(D_*)_\F$, then 
    \[ I(C_*) \leq \|f_k\| \cdot I(D_*).
    \]
  \end{itemize}
\end{defi}

Monotonicity includes both an integrality condition and
quantitative information (namely, $\|f_k\|$). Taking chain
maps as in the monotonicity condition as morphisms on~$\mCh$,
one could view the monotonicity condition as a functoriality
to the poset category~$[0,\infty]$, twisted by quantitative
data.

\begin{prop}[chain extension estimate]\label{prop:sizechainext}
  Let $\F$ be a subfield of~$\R$, let $k \in \N$,
  and let $I$ be a size invariant over~$\F$ in degree~$n$. 
  Let $C_*$ be a marked free $\Z$-chain complex,
  let $s \colon C_k \otimes_\Z \F \to H_k(C_*;\F)$
  be a chain extension of~$C_*$ over~$\F$ in
  degree~$k$, and let $B$ be a cointegral
  $\F$-basis of~$H_k(C_*;\F)$.
  Then, we obtain
  \[ I(C_*)
  \leq \|s\|^B,
  \]
  where $\|\cdot\|^B$ refers to the operator norm taken with respect
  to the $\ell^1$-norms associated with the $\F$-basis 
  of~$C_k \otimes_\Z \F$ induced by the marked $\Z$-basis of~$C_k$
  and $B$ of~$H_k(C_*;\F)$, respectively.
\end{prop}
\begin{proof}
  We apply monotonicity to~$s$ and normalisation: 
  We consider the $\Z$-chain complex~$D_*$ concentrated
  in degree~$k$ given by
  \[ D_k := \bigoplus_B \Z,
  \]
  with the marked $\Z$-basis~$B$.
  The chain extension~$s$ can be viewed as a chain
  map~$f_* \colon C_* \otimes_\Z \F \to D_*\otimes_\Z \F$
  (through the $\F$-isomorphism~$D_k \otimes_\Z \F \to H_k(C_*;\F)$
  induced by the basis~$B$ and extension by the zero
  maps in degrees~$\neq k$). By construction,
  \begin{itemize}
  \item
    $H_k(f_*) \colon H_k(C_*;\F) \to H_k(D_*;\F)$ is an $\F$-isomorphism,
  \item
    cointegrality of~$B$ ensures that $H_k(f_*) (H_k(C_*)_\F ) \subset H_k(D)_\F$,
  \item
    and we have $\|f_k\| = \|s\|^B$.
  \end{itemize}
  Therefore, monotonicity 
  applies to~$f_*$ and the target complex~$D_*$
  falls into the normalisation case. We obtain
  \[
    I(C_*)
    \leq \|f_k\| \cdot I(D_*)
    \leq \|s\|^B \cdot 1.
    \qedhere
  \]
\end{proof}

\begin{example}[vanishing]\label{exa:vanishing}
  Let $\F$ be a subfield of~$\R$, let $k \in \N$,
  let $I$ be a size invariant over~$\F$ in degree~$k$,
  and let $C_*$ be a marked free $\Z$-chain complex
  with~$H_k(C_*;\F) \cong 0$ (e.g., through~$C_k \cong 0$).
  Then $I(C_*) = 0$ by the chain extension estimate
  (Proposition~\ref{prop:sizechainext}), applied to
  the zero chain extension.
\end{example}

\begin{cor}[finiteness]\label{prop:sizefinite}
  Let $\F$ be a subfield of~$\R$, let $k \in \N$,
  let $I$ be a size invariant over~$\F$ in degree~$k$, 
  and let $C_*$ be a marked free $\Z$-chain complex.
  If $\rk_\Z C_k < \infty$, then
  \[ I(C_*) < \infty.
  \]
\end{cor}
\begin{proof}
 We use the chain extension estimate (Proposition~\ref{prop:sizechainext}): 
 Because $\F$ is a field, there exists a chain
 extension~$s \colon C_k\otimes_\Z \F \to H_k(C_*;\F)$
 of~$C_*$ over~$\F$ in degree~$k$.
 Let $B$ be a cointegral basis of~$H_k(C_*;\F)$
 (for instance, induced from a $\Z$-basis of
 the finitely generated free $\Z$-module~$H_k(C_*)_\F$).

 We write~$\Sigma_k \subset C_k$ for the marked $\Z$-basis of~$C_k$.
 Because $C_k \otimes_\Z \F$ is equipped with the $\ell^1$-norm
 given by the finite marked basis~$\Sigma_k$, we obtain
 \[ \|s\|^B
 = \max_{\sigma \in \Sigma_k} \bigl| s(\sigma) \bigr|_1^B
 < \infty.
 \]
 The chain extension estimate (Proposition~\ref{prop:sizechainext})
 shows that $I(C_*) \leq \|s\|^B$
 is finite.
\end{proof}

%%%%%%%%%%%%%%%%%%%%%%%%%%%%%%%%%%%%%%%%%%%%%%%%%%%%%%%%%%%%%%%%%
\section{Cohomological lifting volume}\label{sec:liftingvol}

We introduce the cohomological lifting volume as optimal bound
measuring minimal lifts of cohomology classes from~$H^*(C_*;\R/\Z)$
to~$H^*(C_*;\R)$. We then establish the fundamental properties
of the cohomological lifting volume:
\begin{itemize}
\item it is a size invariant in the sense of Section~\ref{sec:size} and thus
  satisfies normalisation, monotonicity, and the chain extension
  estimate;
\item it gives control on integral cohomological fillings.
\end{itemize}
The applications to finite coverings and simplicial subdivisions
are contained in Section~\ref{subsec:fincov} and
Section~\ref{sec:liftsc}, respectively.

%%%%%%%%%%%%%%%
\subsection{Definition and basic properties}\label{subsec:defliftingvol}

\begin{defi}[cohomological lifting volume]
  Let $C_*$ be a marked free chain complex over~$\Z$,
  let $\F$ be a subfield of~$\R$, and let $k \in \N$.
  We write~$\pi \colon H^k(C_*;\F) \to H^k(C_*;\F/\Z)$
  for the map induced by the canonical
  projection~$\F \to \F/\Z$.
  \begin{itemize}
  \item
    For~$\varphi \in H^k(C_*;\F/\Z)$, we define the
    \emph{lifting volume} of~$\varphi$ by
    \[ L_k^\F(C_*,\varphi)
    := \inf_{\widetilde \varphi \in \pi^{-1}(\varphi) \subset H^k(C_*;\F)}
    \| \widetilde \varphi \|_\infty
    \in [0,\infty].
    \]
  \item
    The \emph{cohomological lifting volume of~$C_*$
      over~$\F$ in degree~$k$} is defined as
    \[ L_k^\F (C_*)
    := \sup_{\varphi \in \im \pi \subset H^k(C_*;\F/\Z)} L_k^\F(C_*,\varphi)
    \in [0,\infty].
    \]
  \end{itemize}
\end{defi}

If $\F$ is a subfield of~$\R$, then $\F$ contains~$\Q$, the field~$\F$
is dense in~$\R$, and $[-1/2,1/2)$ is a fundamental domain in~$\F$
  for~$\F/\Z$.

\begin{rem}\label{rem:UCT}
  Let $C_*$ be a marked free chain complex over~$\Z$,
  let $\F$ be a subfield of~$\R$, and let $k \in \N$.
  By the universal coefficient theorem, 
  because $\F$ and $\F/\Z$ are divisible $\Z$-modules,
  the horizontal evaluation maps in
  the commutative diagram
  \[
  \begin{tikzcd}
    H^k(C_*;\F)
    \ar[d,"\pi", description]
    \ar[r]
    & \Hom_\Z\bigl(H_k(C_*), \F\bigr)
    \ar[d]
    \\
    H^k(C_*;\F/\Z)
    \ar[r]
    & \Hom_\Z\bigl(H_k(C_*), \F/\Z\bigr)
  \end{tikzcd}
  \]
  are isomorphisms over~$\F$ and $\Z$, respectively.
  The upper evaluation map factors
  over the restriction isomorphisms~$\Hom_\F(H_k(C_*;\F), \F)
  \to \Hom_\Z(H_k(C_*)_\F,\F) \to \Hom_\Z(H_k(C_*),\F)$.
\end{rem}

\begin{example}[normalisation]\label{exa:liftnormalisation}
  Let $\F$ be a subfield of~$\R$, let $k \in \N$,
  and let $C_*$ be a marked free $\Z$-chain complex
  with~$\partial_{k+1} =0$.
  Then, we have
  \[
  L_k^\F(C_*)
  = \begin{cases}
    0 & \text{if $C_k \cong_\Z 0$}
    \\
    \frac12 & \text{otherwise}.
  \end{cases}
  \]
  Indeed, this follows from the fact that all cochains in degree~$k$
  are cocycles and that $[-1/2,1/2)$ is a fundamental domain in~$\F$
  for~$\F/\Z$.
\end{example}

\begin{rem}
  The cohomological lifting volume is \emph{not} a chain homotopy
  invariant: For~$d \in \N_{>0}$, we consider the $\Z$-chain
  complex~$C_*$ 
  \[
  \begin{tikzcd}
    \Z
    \ar{r}{(d,1)}
    & 
    \Z \oplus \Z
    \ar{r}{0}
    & \Z,
  \end{tikzcd}
  \]
  concentrated in degrees~$2$, $1$,~$0$ with the marked bases given
  by~$1 \in \Z$ and $(e_1,e_2) := ((1,0), (0,1))$ of~$\Z\oplus
  \Z$. This is related to a cellular chain complex of the mapping
  cylinder of a $d$-fold covering~$S^1 \to S^1$.

  If $f \in \Hom_\Z(C_1,\F)$ is a cocycle, the cocycle
  condition~$0 = \delta^1f (1) = f(\partial_2(1)) = d \cdot f(e_1) + f(e_2)$ forces
  \[ \bigl|f(e_2)\bigl| = \bigl|- d \cdot f(e_1)\bigr|
  = d \cdot \bigl| f(e_1)\bigr|.
  \]
  From this and an explicit description of~$H^1(C_*;\F/\Z)$ and
  $H^1(C_*;\F)$, we can calculate~$L^\F_1(C_*) = d/2$. But the
  chain homotopy equivalence type of~$C_*$ is independent of~$d$.
\end{rem}

\begin{prop}[monotonicity]\label{prop:liftmono}
  Let $\F$ be a subfield of~$\R$, let $k\in \N$, and let $C_*$ and
  $D_*$ be marked free $\Z$-chain complexes.  Let $f_* \colon C_*
  \otimes_\Z \F \to D_* \otimes_\Z \F$ be an $\F$-chain map that
  induces an $\F$-isomorphism~$H_k(f_*) \colon H_k(C_*;\F) \to
  H_k(D_*;\F)$ and that satisfies the integrality condition~$H_k(f_*)
  (H_k(C_*)_\F) \subset H_k(D_*)_\F$. Then, we have
  \[ L_k^\F(C_*) \leq \|f_k\| \cdot L_k^\F(D_*).
  \]
\end{prop}
\begin{proof}
  In view of the natural isomorphism~$\Hom_\Z(\args,\F)
  \cong_\F \Hom_\F ( \args\otimes_\Z \F,\F)$,
  the chain map~$f_*$ induces a cochain map~$\Hom_\Z(D_*,\F)
  \to \Hom_\Z(C_*,\F)$. This cochain map induces, by
  the homological assumption on~$f_*$ and
  the universal coefficient theorem, an 
  $\F$-iso\-morphism~$H^k(f_*) \colon H^k(D_*;\F) \to H^k(C_*;\F)$
  with~$\|H^k(f_*)\| \leq \|f_k\|$. 

  Let $\varphi \in H^k(C_*;\F/\Z)$ be in the image of
  the reduction homomorphism~$\pi$, say~$\varphi = \pi(\varphi')$
  with~$\varphi' \in H^k(C_*;\F)$.
  We show that $L_k^\F(C_*,\varphi)\leq \|f_k\| \cdot L_k^\F(D_*)$.
  Taking the supremum over all~$\varphi$ then completes the proof.
  
  We consider the transported classes 
  \[ \psi' := H^k(f_*)^{-1}(\varphi')
  \in H^k(D_*;\F)
  \qand
  \psi := \pi(\psi') \in H^k(D_*;\F/\Z).
  \]
  and show that $L_k^\F(C_*,\varphi) \leq \|f_k\| \cdot L_k^\F(D_*,\psi)$: 
  Let $\widetilde \psi \in H^k(D_*;\F)$ with~$\pi(\widetilde\psi) = \psi$.
  We set
  \[ \widetilde \varphi := H^k(f_*)(\widetilde \psi) \in H^k(C_*;\F).
  \]
  Then, we obtain:
  \begin{itemize}
  \item By construction,
    $ \|\widetilde \varphi\|_\infty
    \leq \bigl\| H^k(f_*) \bigr\| \cdot \|\widetilde \psi\|_\infty
    \leq \|f_k\| \cdot \|\widetilde\psi\|_\infty.
    $
  \item Moreover, $\pi(\widetilde \varphi) = \varphi$, which can
    be seen as follows: In view of the universal coefficient theorem
    and $\pi(\varphi') = \varphi$,
    we only need to show that
    \[ \fa{\alpha \in H_k(C_*)}
    \langle \widetilde \varphi,\alpha\rangle
    - \langle \varphi',\alpha\rangle
    \in \Z
    \]
    (Remark~\ref{rem:UCT}). 
    Let $\alpha \in H_k(C_*)$. Then, we have
    \begin{align*}
      \langle \widetilde \varphi,\alpha \rangle
      - \langle \varphi',\alpha\rangle
      & 
      = \langle \widetilde \varphi - \varphi' , \alpha \rangle
      = \bigl\langle
        H^k(f_*)(\widetilde \psi - \psi'),\alpha
        \bigr\rangle
      \\
      & 
      = \bigl\langle
        \widetilde \psi - \psi', H_k(f_*)(\alpha)
        \bigr\rangle
        .
    \end{align*}
    this difference lies in~$\Z$, because 
    $\pi(\widetilde \psi) = \psi = \pi(\psi')$
    and $H_k(f_*)(\alpha) \in H_k(D_*)_{\F}$ by
    the integrality condition~$H_k(f_*) (H_k(C_*)_\F)
    \subset H_k(D_*)_\F$.
  \end{itemize}
  We conclude that $L_k^\F(C_*,\varphi) \leq \|f_k\| \cdot \|\widetilde \psi\|_\infty$. 
  Taking the infimum over all~$\widetilde \psi$, we obtain~$L_k^\F(C_*,\varphi)
  \leq \|f_k\| \cdot L_k^\F(D_*,\psi) \leq \|f_k\| \cdot L_k^\F(D_*)$.
\end{proof}

\begin{prop}\label{prop:liftsize}
  Let $\F$ be a subfield of~$\R$ and let $k \in \N$. Then, $L_k^\F$
  (more precisely, even $2 \cdot L_k^\F$) is a size invariant
  over~$\F$ in degree~$k$ in the sense of Section~\ref{sec:size}.
\end{prop}
\begin{proof}
  The cohomological lifting volume~$L_k^\F$ and $2 \cdot L_k^\F$ satisfy 
  \begin{itemize}
  \item normalisation by Example~\ref{exa:liftnormalisation} and 
  \item monotonicity by Proposition~\ref{prop:liftmono}.
  \qedhere
  \end{itemize}
\end{proof}

In particular, we obtain the corresponding version of
the chain extension estimate and finiteness. For convenience,
we spell them out explicitly:

\begin{cor}[chain extension estimate]\label{cor:liftchainext}
  Let $\F$ be a subfield of~$\R$, let $k \in \N$,
  let $C_*$ be a marked free $\Z$-chain complex,
  let $s \colon C_k \otimes_\Z \F \to H_k(C_*;\F)$
  be a chain extension of~$C_*$ over~$\F$ in
  degree~$k$, and let $B$ be a cointegral
  $\F$-basis of~$H_k(C_*;\F)$.
  Then, we obtain
  \[ L_k^\F(C_*)
  \leq \frac12 \cdot \|s\|^B,
  \]
  where $\|\cdot\|^B$ refers to the operator norm taken with respect
  to the $\ell^1$-norms associated with the $\F$-basis 
  of~$C_k \otimes_\Z \F$ induced by the marked $\Z$-basis of~$C_k$
  and $B$ of~$H_k(C_*;\F)$, respectively.
\end{cor}
\begin{proof}
  Because $2 \cdot L_k^\F$ is a size invariant over~$\F$
  in degree~$k$, this is an instance of the general
  chain extension estimate (Proposition~\ref{prop:sizechainext}).   
\end{proof}

\begin{cor}[finiteness]\label{cor:liftfin}
  Let $\F$ be a subfield of~$\R$, let $k \in \N$,
  and let $C_*$ be a marked free $\Z$-chain complex.
  If $\rk_\Z C_k < \infty$, then
  \[ L_k^\F(C_*) < \infty.
  \]
\end{cor}
\begin{proof}
  Because $L_k^\F$ is a size invariant over~$\F$
  in degree~$k$, this is an instance of the general
  finiteness property (Proposition~\ref{prop:sizefinite}).   
\end{proof}

%%%%%%%%%%%%%%%
\subsection{Integral cohomological fillings}\label{subsec:liftfill}

The cohomological lifting volume gives upper bounds for
integral cohomological filling problems. 

\begin{prop}[integral cohomological fillings]\label{prop:intcohfill}
  Let $\F$ be a subfield of~$\R$, let $k \in \N$,
  and let $C_*$ be a marked free $\Z$-chain complex.
  Let $f \in \Hom_\Z(C_k,\F)$ be a cochain such
  that $\delta f$ is a coboundary in the cochain complex~$\Hom_\Z(C_*,\Z)$.
  Then, for every~$\varepsilon \in \R_{>0}$, there exists
  a cochain~$\widetilde f \in \Hom_\Z(C_k,\Z)$ with
  \[ \delta \widetilde f = \delta f
  \qand
  |\widetilde f - f|_\infty
  \leq \frac12 \cdot \|\partial_k\| + L_k^\F(C_*) + \varepsilon.
  \]
  Here, $|\cdot|_\infty$ and $\|\cdot\|$ refer to the norms induced by the
  marked $\Z$-bases of~$C_k$ and~$C_{k-1}$, respectively.
\end{prop}
\begin{proof}
  This is a straightforward translation of the corresponding
  statement in the simplicial context~\cite[Lemma~3.9]{CDMW}
  to our setting: 
  Let $\pi$ denote the maps on the level of cochains and cohomology
  induced by the canonical projection~$\F \to \F/\Z$.

  As a first step, 
  we consider the cochain~$\pi(f) \in \Hom_\Z(C_k,\F/\Z)$.
  Because $\delta f$ is $\Z$-valued (as coboundary in~$\Hom_\Z(C_*,\Z)$),
  the cochain~$\pi(f)$ in fact is a cocycle. We define
  \[ \varphi := \bigl[\pi(f)\bigr] \in H^k(C_*;\F/\Z). 
  \]
  We show that $\varphi$ lies in the image of~$\pi \colon H^k(C_*;\F) \to H^k(C_*;\F/\Z)$: 
  By hypothesis, there is a cochain~$f' \in \Hom_\Z(C_k,\Z)$ with
  \[ \delta f = \delta f'.
  \]
  Then $f-f' \in \Hom_\Z(C_k,\F)$ is a cocycle and $f - f'$  
  lifts~$\pi(f)$. Therefore,
  $\pi([f-f']) = [\pi(f)] = \varphi$, which shows that $\varphi$ is
  in the image of~$\pi$. 

  Let $\varepsilon \in \R_{>0}$.  By definition of the lifting volume,
  there exists a class~$\widetilde \varphi \in H^k(C_*;\F)$
  with~$\pi(\widetilde \varphi) = \varphi$ that admits a representing
  cocycle~$g \in \Hom_\Z(C_k,\F)$ with~$|g|_\infty \leq L_k^\F(C_*) +
  \varepsilon$. 
  The cochain~$f - g$ fills~$\delta f$,
  but might not be $\Z$-valued; this can be adjusted as follows:
  Because $\pi(g)$ and $\pi(f)$ represent the same class~$\varphi$, there
  exists a cochain~$h \in \Hom_\Z(C_{k-1},\F/\Z)$ with~$\delta h =
  \pi(g) - \pi(f)$.  We pick a lift~$\widetilde h
  \in\Hom_\Z(C_{k-1},\F)$ of~$h$ with~$|\widetilde h|_\infty \leq 1/2$
  and finally set
  \[ \widetilde f := f - g + \delta \widetilde h
  \in \Hom_\Z(C_k,\F).
  \]
  Then $\widetilde f$ has the desired properties:
  \begin{enumerate}
  \item The cochain~$\widetilde f$ is $\Z$-valued, because
    \[ \pi(\widetilde f)
    = \pi(f - g + \delta \widetilde h)
    = \pi(f) - \pi(g) + \pi(\delta \widetilde h)
    = \pi(f) - \pi(g) + \delta h
    = 0.
    \]
  \item We have~$\delta \widetilde f = \delta f$, because
    \[ \delta \widetilde f
    = \delta (f - g + \delta \widetilde h)
    = \delta f - \delta g + \delta \circ \delta (\widetilde h)
    = \delta f - 0 + 0.
    \]
  \item
    We have
    \[ | \widetilde f - f |_\infty
    = |-g + \delta^{k-1} \widetilde h|_\infty
    \leq |g|_\infty + \|\delta^{k-1}\| \cdot | \widetilde h|_\infty
    \leq L_k^\F(C_*) + \varepsilon + \|\partial_k\| \cdot \frac12.
    \qedhere
    \]
  \end{enumerate}
\end{proof}

%%%%%%%%%%%%%%%
\subsection{Integral cohomological fillings on finite coverings}\label{subsec:fincov}

We reformulate the integral cohomological filling estimate in terms of
the expansion terminology of Bader and
Sauer~\cite[Section~2]{bader_sauer} and give examples in the setting
of finite coverings.

\begin{defi}[expansion constant]
  Let $R$ be a normed ring (e.g., a subring of~$\R$),
  let $M$, $N$ be normed $R$-modules, 
  and let $f \colon M \to N$ be $R$-linear.
  \begin{itemize}
  \item For~$y \in \im f$, we define the \emph{expansion constant of~$f$
    at~$y$} by
    \[ \Xi^R(f,y)
    := \inf \bigl\{ c \in \R_{\geq 0}
    \bigm| \exi{x \in f^{-1}(y) \subset M} |x| \leq c \cdot |y|
    \bigr\}.
    \]
  \item
    The \emph{expansion constant of~$f$} is 
    \[ \Xi^R(f)
    := \sup_{y \in \im f} \Xi_R(f,y)
    \in [0,\infty].
    \]
  \item
    If $C_*$ is a marked free $\Z$-chain complex and $k \in \N$,
    we equip the dual chain complex with the respective
    $|\cdot|_\infty$-norms and abbreviate  
    \[ \Xi_k^R(C_*)
    := \Xi^R \bigl( \Hom_\Z(\partial_{k+1}, R)
    \colon \Hom_\Z(C_k,R) \to \Hom_\Z (C_{k+1},R) \bigr)
    \in [0,\infty].
    \]
  \end{itemize}
\end{defi}

\begin{cor}[integral cohomological fillings via expansion constants]\label{cor:exp}
  Let $\F$ be a subfield of~$\R$, let $k \in \N$, and let $C_*$ be a
  marked free $\Z$-chain complex. Then:
  \[ \Xi_k^\Z(C_*)
  \leq \Xi_k^\F(C_*) + \frac12 \cdot \|\partial_k\| + L_k^\F(C_*).
  \]
\end{cor}
\begin{proof}
  We only need to unfold the definition of the expansion constant and apply
  the integral cohomological filling estimate
  (Proposition~\ref{prop:intcohfill}): 

  Without loss of generality, we may assume that $\Xi_k^\F(C_*)$
  is finite. 
  Let $c \in \R$ with~$c > \Xi_k^\F(C_*)$. Let $g \in \Hom_\Z(C_k,\Z)$
  be in the image of the coboundary operator~$\delta_\Z$ on~$\Hom_\Z(C_*,\Z)$,
  without out loss of generality let~$g \neq 0$.
  Then, there exists a cochain~$f \in \Hom_\Z(C_*,\F)$ with
  \[ \delta_\F f = g
  \qand
  |f|_\infty \leq c \cdot |g|_\infty.
  \]
  Given~$\varepsilon \in \R_{>0}$, by the integral cohomological filling
  estimate (Proposition~\ref{prop:intcohfill}), there is an integral
  cochain~$\widetilde f \in \Hom_\Z(C_*,\Z)$ with
  $\delta_\Z \widetilde f = \delta_\F f = g$ 
  and
  \begin{align*}
    |\widetilde f |_\infty
    & \leq
    |f|_\infty + \frac12 \cdot \|\partial_k\| + L_k^\F(C_*) + \varepsilon
    \\
    & \leq c \cdot |g|_\infty
    + \frac12 \cdot \|\partial_k\| + L_k^\F(C_*) + \varepsilon
    \\
    & \leq \Bigl( c + \frac12 \cdot \|\partial_k\| + L_k^\F(C_*) + \varepsilon \Bigr)
    \cdot |g|_\infty;
  \end{align*}
  for the last step, we used integrality of~$g$: If $g \neq 0$,
  then $|g|_\infty \geq 1$. 
  Taking $\varepsilon \to 0$ and the infimum over all~$c > \Xi_k^\F(C_*)$
  proves the claim.
\end{proof}

We denote the cellular chain complex functor by~$\ccell*$ (marked by
the bases of open cells) and implicitly use that cellular homology and
singular homology are naturally isomorphic.

\begin{cor}\label{cor:fincov}
  Let $X$ be a finite CW-complex, let $p \colon \overline X \to X$ be a
  finite sheeted covering map, and let $\overline X$ be equipped with
  the CW-structure induced by~$p$. Let $k \in \N$, let $\F$
  be a subfield of~$\R$  
  and assume that $\dim_\F H_k(\overline X;\F) = \dim_\F H_k(X;\F)$.
  \begin{enumerate}
  \item
    Then:
    \[ L_k^\F\bigl( \ccell*(\overline X) \bigr)
    \leq L_k^\F\bigl( \ccell*(X) \bigr)
    \]
  \item
    In particular, denoting the cellular boundary operator
    of the universal covering~$\widetilde X$
    of~$X$ by~$\widetilde \partial_*$,
    we obtain:
    \[ \Xi_k^\Z\bigl(\ccell*(\overline X) \bigr)
    \leq \Xi_k^\F\bigl(\ccell*(\overline X)\bigr)
    + \frac12 \cdot \|\widetilde \partial_k\| + L_k^\F\bigl(\ccell* (X)\bigr)
    \]
  \end{enumerate}
\end{cor}
\begin{proof}
  % # sheets non-zero
  \emph{Ad~1.}
  This is an instance of monotonicity of the lifting volume: 
  Transfer shows that the induced map~$H_k(p ;\F) \colon H_k(\overline X;\F)
  \longrightarrow H_k(X;\F)$ is surjective. As both $\F$-vector spaces
  are finite-dimensional (both CW-complexes are finite) and of the
  same dimension (by assumption), $H_k(p;\F)$ is an $\F$-isomorphism.
  The integrality
  condition is satisfied because $H_k(p ;\F)$ is induced by a continuous map. 
  Therefore, monotonicity (Proposition~\ref{prop:liftmono})
  implies that $L_k^\F(\ccell*(\overline X))
  \leq \| \ccell k(p)\| \cdot L_k^\F(\ccell*(X)) \leq 1 \cdot L_k^\F(\ccell*(X))$. 

  \emph{Ad~2.}
  We combine the estimate of the first part with
  Corollary~\ref{cor:exp} and the straightforward
  estimate~$\|\partial_k^{\cell,\overline X}\| \leq \|\widetilde \partial_k\|$.
\end{proof}

\begin{rem}[uniformity]\label{rem:expansionindep}
  The point of the second part of Corollary~\ref{cor:fincov}
  is that $\frac12 \cdot \|\widetilde \partial_k\| + L_k^\F(\ccell* (X))$ 
  does not depend on the finite covering~$\overline X$, but only on
  the base complex~$X$ (and its universal covering). Also, because $X$
  is a finite CW-complex, both $\|\widetilde \partial_k\|$ and
  $L_k^\F(\ccell*(X))$ are finite.
\end{rem}

\begin{example}\label{exa:kazhdan}
  Let $X$ be a finite connected CW-complex whose fundamental
  group~$\Gamma$ is Kazhdan. Then, there exists a constant~$c \in \R_{>0}$
  such that: For all finite coverings~$\overline X \to X$
  (with the induced CW-structure), we have 
  \[ \Xi_0^\Z \bigl(\ccell * (\overline X)\bigr)
  \leq c
  \qand
  \Xi_1^\Z \bigl(\ccell * (\overline X)\bigr)
  \leq c.
  \]
  Indeed, we only need to consider the case that $\overline X$
  is connected. 
  As a first step, we apply Corollary~\ref{cor:fincov}
  in this situation. The homology condition is satisfied
  in degree~$0$, because $X$ and $\overline X$ are connected
  and thus~$H_0(\overline X;\R) \cong_\R \R \cong_\R H_0(X;\R)$.
  The homology condition is satisfied in degree~$1$, because
  $\Gamma$ is Kazhdan, and so
  $H_1(X;\R) \cong_\R H_1(\Gamma;\R) \cong_\R 0$; the same reasoning
  applies to~$H_1(\overline X;\R)$, because the fundamental group
  of~$\overline X$ is Kazhdan as finite index subgroup of the
  Kazhdan group~$\Gamma$. 

  We then conclude via the observation on uniformity in
  Remark~\ref{rem:expansionindep}, and the result of Bader and Sauer
  showing that $\Xi_0^\R(\ccell * (\overline X))$ and $\Xi_1^\R(\ccell
  * (\overline X))$ are uniformly bounded in this
  situation~\cite[Theorem~2.8]{bader_sauer}.
 
  This estimate is slightly weaker than the version of Bader and
  Sauer~\cite[Theorem~3.13]{bader_sauer}, which gives equality between
  the integral and the real expansion constants. However, for their
  proofs of waist inequalities, the above version would be sufficient.
\end{example}

Corollary~\ref{cor:fincov} also admits a corresponding version
for the singular chain complex (and more general spaces). However,
in general, lifting volumes of singular chain complexes will not
be finite. 

\begin{example}
  Let $\overline M \to M$ be a finite sheeted covering of
  oriented closed connected manifolds. We write 
  $n := \dim M = \dim \overline M$ and equip $M$ with
  a finite CW-structure. Let $\F$ be a subfield of~$\R$. 
  Then $\dim_\F H_n(\overline M;\F) = 1 = \dim_\F H_n(M;\F)$
  and so Corollary~\ref{cor:fincov} applies in this situation
  (both to the cellular chain complex and the singular chain complex).
  We will return to this in Example~\ref{exa:fincovsv}.
\end{example}

For general finite sheeted covering maps, the cohomological
lifting volume does not satisfy monotonicity (as can, e.g.,
be seen in degree~$2$ for the double covering~$S^2 \to \R P^2$).
However, it is unknown, whether the cohomological lifting
volume is uniformly bounded over a given base:

\begin{question}
  Let $X$ be a finite CW-complex and~$k \in \N$. Does there
  exist a constant~$c\in \R_{\geq 0}$ with the following property?
  For every finite sheeted covering~$\overline X \to X$ (where $\overline X$
  is equipped with the induced CW-structure), we have~$L_k^\F\bigl( \ccell*(\overline X) \bigr)
  \leq c$.
\end{question}

%%%%%%%%%%%%%%%%%%%%%%%%%%%%%%%%%%%%%%%%%%%%%%%%%%%%%%%%%%%%%%%%%
\section{Cohomological lifting %and integral cohomological filling
  for subdivisions}\label{sec:liftsc}

Monotonicity of the cohomological lifting volume gives uniform control for
cohomological liftings and fillings on subdivisions of simplicial
complexes (Theorem~\ref{thm:liftsubdiv}, Corollary~\ref{cor:fillsubdiv}). 
As a preparation, we first recall basics on simplicial subdivisions.

We fix notation on simplicial complexes, in the combinatorial
sense: A \emph{simplicial complex} is a set~$X$ of finite sets, closed
under taking subsets. The elements of~$X$ are called
\emph{simplices}. If $\sigma \in X$, we write $\dim \sigma := \#
\sigma -1$.  For~$n\in \N$, we write~$X_n := \{ \sigma \in X \mid \dim
\sigma = n\}$ for the set of all $n$-simplices.  A \emph{simplicial
  map}~$X \to Y$ between simplicial complexes~$X$,~$Y$ is a map~$f
\colon X_0 \to Y_0$ with
\[ \fa{\sigma \in X} f(\sigma) \in Y.
\]
The geometric realisation of simplicial complexes is denoted
by~$|\cdot|$ and we extend this notation to simplices.

We denote the simplicial chain complex and simplicial homology
functors on simplicial complexes by $C_*^\Delta$ and $H_*^\Delta$,
respectively. We view the simplicial chain complex as marked through
the marked bases given by the simplices in the respective
dimension; more precisely, to avoid sign ambiguities, we
pick an ordering on the set of vertices and consider ordered
simplices.

%%%%%%%%%%%%%
\subsection{Simplicial subdivisions}\label{subsec:subdiv}

We recall the definition of subdivisions of simplicial complexes
and the key fact that each subdivision admits a simplicial
map to the original simplicial complex that induces an isomorphism
in homology.

\begin{defi}[subdivision of a simplicial complex]
  Let $X$ be a simplicial complex. A \emph{subdivision of~$X$}
  is a pair~$(X',i)$, consisting of a simplicial complex~$X'$
  and a homeomorphism~$i \colon |X'| \to |X|$ with the following
  properties:
  \begin{itemize}
  \item We have $X_0 \subset X'_0$ and $i|_{X_0} = \id_{X_0}$.
  \item If $\sigma \in X$, then for the subcomplex 
    $Y:=
    \{ \tau
    \mid \tau \in X',\
    i(|\tau|) \subset |\sigma|
    \}
    $ 
    of~$X'$, the restriction map~$i|_{|Y|} \colon |Y| \to |\sigma|$
    is a homeomorphism and moreover 
    $Y = \{ \tau \mid \tau \in X',\ i(\tau) \subset |\sigma|\}$
    (convexity).
  \end{itemize}
 
  Usually, we suppress~$i$ in the notation; we then speak of
  a subdivision~$X'$ of~$X$ and write~$|\tau| \subset |\sigma|$
  if $\tau$ and $\sigma$ are simplices of~$X'$ and~$X$, respectively,
  satisfying $i(|\tau|) \subset |\sigma|$.  
\end{defi}

As a consequence of this definition, for each simplex in the given
subdivision~$X'$, its geometric realisation is contained in the
geometric realisation of some simplex of the original simplicial
complex~$X$. Moreover, the simplices of~$X$ induce a
corresponding structure of (finite) subcomplexes of~$X'$.

An example of a systematic simplicial subdivision is the barycentric
subdivision.

\begin{lemma}\label{lem:subdivquasiiso}
  Let $X$ be a finite simplicial complex and let $X'$ be a subdivision
  of~$X$. Then, there exists a simplicial map~$f \colon X' \to X$
  such that: For every field~$\F$ and every~$k \in \N$, the induced
  map~$H_k^\Delta(f;\F) \colon H_k^\Delta(X';\F) \to H_k^\Delta(X;\F)$
  is an $\F$-isomorphism.
\end{lemma}
\begin{proof}
  We pick a total ordering on the vertices of~$X'$ (and hence
  also on the vertices and faces of~$X$). For each vertex~$x \in X'_0$,
  we write~$\sigma_x$ for the support simplex of~$x$ in~$X$,
  i.e., for the smallest simplex~$\sigma_x$ in~$X$ such that
  $|x|$ is contained in the relative interior of~$|\sigma_x|$.
  We define 
  \begin{align*}
    f \colon X'_0
    & \to X_0
    \\
    x
    & \mapsto \max \sigma_x
  \end{align*}
  and show the following:
  \begin{enumerate}
  \item\label{subdiv.1}
    The map~$f$ is a simplicial map~$X' \to X$.
  \item\label{subdiv.2}
    If $\sigma$ is a simplex of~$X$ and $n := \dim \sigma$,
    then there exists precisely one $n$-simplex~$\sigma'$
    of~$X'$ with
    \[
    |\sigma'|\subset |\sigma|
    \qand 
    f(\sigma') = \sigma.
    \]
    All other $n$-simplices~$\tau$ of~$X'$ with~$|\tau| \subset |\sigma|$
    satisfy~$\dim f(\tau) < n$.
  \item\label{subdiv.3}
    There exists a chain map~$g \colon C_*^\Delta(X) \to C_*^\Delta(X')$
    with~$C_*^\Delta(f) \circ g = \id_{C_*^\Delta(X)}$.
  \end{enumerate}
  Once these properties are established, we can apply
  Lemma~\ref{lem:finquasiiso} to conclude that $H_k^\Delta(f;\F)$ is an
  isomorphism for every field~$\F$ and every~$k \in \N$.
    
  Instead of invoking Lemma~\ref{lem:finquasiiso}, one could also
  geometrically construct a chain homotopy~$g \circ C_*^\Delta(f)
  \simeq \id_{C_*^\Delta(X')}$

 \emph{Ad~\ref{subdiv.1}.}
  Because $X'$ is a subdivision of~$X$, for every simplex~$\tau =
  \{x_0,\dots,x_n\} \in X'_n$, there exists a simplex~$\sigma$
  of~$X$ with~$\sigma_{x_0}, \dots, \sigma_{x_n} \subset \sigma$
  (every simplex~$\sigma$ with~$|\tau|\subset |\sigma|$ will do).
  Hence, also
  \[ f(\tau)
  = \bigl\{ f(x_0), \dots, f(x_n)\bigr\}
  = \{\max \sigma_{x_0}, \dots, \max \sigma_{x_n}\}
  \subset \sigma,
  \]
  and so $f(\tau)$ is a simplex of~$X$.
  Thus, $f$ is a simplicial map~$X' \to X$.

  \emph{Ad~\ref{subdiv.2}.}
  This can be shown by induction on the dimension~$n$: 
  The conditions are clearly satisfied for~$n = 0$. For
  the induction step, let $n \in\N$ and let the claim
  be shown in dimension~$n$. Let $\sigma \in X_{n+1}$.

  We first show existence of~$\sigma' \in X'_{n+1}$ with $|\sigma'|
  \subset |\sigma|$ and $f(\sigma') = \sigma$. By induction, for the
  $0$-face~$\sigma_0$ of~$\sigma$, there exists a unique $n$-simplex~$\sigma'_0 \in X'$
  with~$|\sigma'_0| \subset |\sigma_0|$ and $f(\sigma'_0) = \sigma_0$.
  Because $X'$ is a simplicial subdivision of~$X$, for
  topological reasons, there is exactly one $\sigma' \in X'_{n+1}$
  such that $\sigma'_0$ is a face of~$\sigma'$ and $\sigma'$
  has a vertex~$v$ in~$|\sigma|\setminus |\sigma'_0|$.
  By construction, $f(v)$ is the $0$-vertex of~$\sigma$ and
  so~$f(\sigma') = \sigma$.

  If $\tau \in X'_{n+1}$ with~$f(\tau) = \sigma$, then all vertices
  of~$\tau$ lie in~$|\sigma|$ (by construction of~$f$) and so $|\tau|
  \subset |\sigma|$ (by convexity). Thus, uniqueness of~$\sigma'$
  follows from the next property.

  Let $\tau \in X'$ with~$|\tau| \subset |\sigma|$ and $\dim f(\tau) = n+1$.
  Then, $\tau = \sigma'$. Indeed, by induction, there is only one
  $n$-simplex~$\sigma'_0 \in X'$ with~$f(\sigma'_0) = \sigma_0$. Hence,
  $\sigma'_0 \subset \tau$. Moreover, because $|\tau|\subset |\sigma|$,
  we have~$f(\tau) \subset \sigma$. Thus, $\dim f(\tau) = n+1$ implies
  that $f(\tau) = \sigma$ and so $\tau$ must have a vertex
  in~$|\sigma| \setminus |\sigma_0|$. But then $\tau$ must coincide
  with~$\sigma'$, as constructed in the first part of the induction
  step.
  
  \emph{Ad~\ref{subdiv.3}.}
  We consider the map
  \begin{align*}
    g \colon C_*^\Delta(X)
    & \to C_*^\Delta(X')
    \\
    X_n \ni \sigma
    & \mapsto
    \sum_{\tau \in X'_n} \varepsilon_\tau \cdot \tau,
  \end{align*}
  where $\varepsilon_\tau \in \{1,-1\}$ corresponds to the
  orientation of~$|\tau|$ relative to~$|\sigma|$.

  Inductively, we see that for each~$\sigma \in X_n$,
  the simplex~$\sigma' \in X'_n$ from Step~\ref{subdiv.2}
  satisfies~$\varepsilon_{\sigma'} =1$, and so the properties
  from Step~\ref{subdiv.2} guarantee that
  \[ C_*^\Delta(f) \circ g(\sigma)
  = 1 \cdot \sigma + 0
  = \id_{C_*^\Delta(X)}(\sigma).
  \qedhere
  \]
\end{proof}

\begin{lemma}\label{lem:finquasiiso}
  Let $X$, $Y$ be finite simplicial complexes, let $k \in \N$, and let
  $\F$ be a field with~$H_k(X;\F) \cong_\F H_k(Y;\F)$.  Let $f \colon
  Y \to X$ be a simplicial map that admits a chain map~$g \colon
  C_*^\Delta(X) \to C_*^\Delta(Y)$ with~$C_*^\Delta (f) \circ g \simeq
  \id_{C_*^\Delta(X)}$. Then $H_*^\Delta(f;\F) \colon H_*^\Delta(Y;\F) \to H_*^\Delta(X;\F)$
  is an $\F$-isomorphism.
\end{lemma}
\begin{proof}
  Every chain homotopy~$C_*^\Delta(f) \circ g \simeq
  \id_{C_*^\Delta(X)}$ induces a chain homotopy~$C_*^\Delta(f;\F)
  \circ (g \otimes_\Z \id_\F) \simeq \id_{C_*^\Delta(X;\F)}$.  In
  particular, $H_k^\Delta(f;\F) \colon H_k^\Delta(Y;\F) \to
  H_k^\Delta(X;\F)$ is surjective. Because $H_*^\Delta(X;\F)$ and
  $H_*^\Delta(Y;\F)$ are finite-dimensional $\F$-vector spaces of the
  same dimension, $H_k^\Delta(f;\F)$ is an isomorphism.
\end{proof}

%%%%%%%%%%%%%
\subsection{Cohomological lifting/filling for subdivisions}\label{subsec:fillsubdiv}

We apply monotonicity of the cohomological lifting volume
to simplicial subdivisions to obtain uniform estimates
for cohomological liftings and fillings. 
In contrast with the original proof for estimates of this
type~\cite[Section~3]{CDMW}, the purely algebraic argument in
Theorem~\ref{thm:liftsubdiv} uses the flexibility of the algebraic
setting to bypass an intricate explicit construction of suitable
spanning trees. Moreover, this argument is robust enough to generalise
in a uniform way to all simplicial subdivisions:

\begin{thm}[cohomological lifting for subdivisions]\label{thm:liftsubdiv}
  Let $X$ be a finite simplicial complex, let $X'$ be a simplicial
  subdivision of~$X$, let $k \in \N$, and let $\F$ be a subfield of~$\R$. Then
  \[ L_k^\F\bigl(C_*^\Delta(X')\bigr)
  \leq L_k^\F\bigl(C_*^\Delta(X)\bigr).
  \]
\end{thm}

The right-hand side is finite by Corollary~\ref{cor:liftfin}.

\begin{proof}
  Because $X'$ is a subdivision of~$X$, 
  there exists a simplicial map~$f \colon X' \to X$
  such that $H_k^\Delta(f;\F) \colon H_k^\Delta(X';\F) \to H_k^\Delta(X;\F)$
  is an isomorphism (Lemma~\ref{lem:subdivquasiiso}).
  Moreover, $f$ being simplicial implies that
  $\|C_k^\Delta(f;\F)\| \leq 1$ (with respect to the marked
  bases given by the simplices of $X'$ and~$X$, respectively)
  and that $H_k^\Delta(f;\F)$ maps integral classes to integral classes.   
  Monotonicity of the lifting volume (Proposition~\ref{prop:liftmono})
  thus gives
  \[ L_k^\F\bigl(C_*^\Delta(X')\bigr)
  \leq \bigl\|C_k^\Delta(f;\F) \| \cdot L_k^\F\bigl(C_*^\Delta(X)\bigr)
  \leq 1 \cdot L_k^\F\bigl(C_*^\Delta(X)\bigr).
  \qedhere
  \]
\end{proof}

\begin{cor}[integral cohomological fillings for subdivisions~\protect{\cite[Lem\-ma~3.9]{CDMW}}]
  \label{cor:fillsubdiv}
  Let $X$ be a finite simplicial complex, let $k \in \N$,
  and let $\F$ be a subfield of~$\R$. Then there exists
  a constant~$c \in \R_{>0}$ such that:
  If $X'$ is a simplicial subdivision of~$X$ and
  $f \in C_\Delta^k(X';\F)$ is a simplicial cochain such that $\delta f$ is a coboundary
  in~$C_\Delta^*(X';\Z)$, then there exists a cochain~$\widetilde f \in C_\Delta^k(X';\Z)$
  with
  \[ \delta \widetilde f = \delta f
  \qand
  |\widetilde f|_\infty \leq |f|_\infty +c.
  \]
\end{cor}
\begin{proof}
  We combine the lifting volume estimate for subdivisions (Theorem~\ref{thm:liftsubdiv})
  with the integral cohomological filling estimate (Proposition~\ref{prop:intcohfill}). 
  We obtain a cochain~$\widetilde f \in C_\Delta^*(X';\Z)$
  satisfying~$\delta \widetilde f = \delta f$ and 
  \begin{align*}
    |\widetilde f|_\infty
    & \leq |f|_\infty + \frac12 \cdot \|\partial_k^{\Delta,X'}\|
    + L_k^\F\bigl(C_*^\Delta(X')\bigr)
    + 1
    \\
    & \leq |f|_\infty +\frac12 \cdot (k+1) + L_k^\F\bigl(C_*^\Delta(X) \bigr) + 1.
  \end{align*}
  The term~$1/2 \cdot (k+1) + L_k^\F(C_*^\Delta(X)) + 1$ is finite
  (Corollary~\ref{cor:liftfin}) and depends only on~$k$ and~$X$.
\end{proof}

In Section~\ref{subsec:gnarlgeo}, we relate this approach
to the original one via spanning trees. 

%%%%%%%%%%%%%%%%%%%%%%%%%%%%%%%%%%%%%%%%%%%%%%%%%%%%%%%%%%%%%%%%%
\section{Gnarledness of chain complexes}\label{sec:gnarl}

We introduce a definition of gnarledness of chain complexes, inspired
by the notion of gnarledness of simplicial complexes by Chambers,
Dotterrer, Manin, and Weinberger~\cite{CDMW}. This algebraic gnarledness
is based on chain extensions and measures how close a chain complex
is to a chain complex with trivial boundary operator and so, by design,
is universal among size invariants. On the one hand,  
the chain extension estimate relates the gnarledness of chain complexes
to the cohomological lifting volume. On the other hand, spanning trees give 
chain extensions, which provides a connection with the original,
geometric, notion of gnarledness of simplicial complexes. Furthermore,
we may express the simplicial volume in terms of gnarledness of
the singular chain complex.

%%%%%%%%%%%%%%%%%
\subsection{Definition and basic examples}

We define the gnarledness of chain complexes in terms of the minimal
``size'' of chain extensions.

\begin{defi}[gnarledness]\label{def:gnarl}
  Let $C_*$ be a marked free $\Z$-chain complex, 
  let $k \in \N$, and let $\F$ be a subfield of~$\R$. We write~$B_k^\F(C_*)$
  for the set of all cointegral $\F$-bases of~$H_k(C_*;\F)$. 
  \begin{itemize}
  \item
    If $s$ is a chain extension of~$C_*$ in degree~$k$ over~$\F$,
    we set
    \[ G_k^\F(C_*,s) := \inf_{B \in B_k^\F(C_*)} \|s\|^B.
    \]
  \item
    The \emph{gnarledness of~$C_*$} in degree~$k$ over~$\F$ is
    defined as
    \[ G_k^\F(C_*)
    := \inf
    \bigl\{ G_k^\F(C_*,s)
    \bigm| \text{$s$ is a chain extension of~$C_*$ in degree~$k$ over~$\F$}
    \bigr\}.
    \]
  \end{itemize}
\end{defi}

\begin{example}
  We consider the $\Z$-chain complex~$C_*$ given by
  \[
  \begin{tikzcd}
    \Z \ar{r}{0}
    & \Z,
  \end{tikzcd}
  \]
  concentrated in degrees~$1$,~$0$ with the marked bases given
  by~$1 \in \Z$. This is the cellular chain complex of~$S^1$
  with respect to the CW-structure consisting of a single $0$-cell
  and a single $1$-cell.

  We have~$G_1^\F(C_*) = 1$ for every subfield~$\F$ of $\R$, which can be
  calculated as follows: Because $\partial_1 = 0$, there is only
  one chain extension~$s$ of~$C_*$ in degree~$1$ over~$\F$, namely
  the canonical projection map
  \[ s \colon C_1 \otimes_\Z \F = \ker (\partial_1\otimes_\Z \id_\F)
  \to H_1(C_*;\F),
  \]
  which is an isomorphism as $\partial_2 = 0$.
  Let $B = (b) \in B_1^\F(C_\ast)$. Then the coefficient~$\lambda \in \F$
  with~$s(1) = \lambda \cdot b$ is a non-zero integer. Therefore,
  \[ \|s\|^B = |\lambda| \geq 1.
  \]
  Conversely, $(s(1)) \in B_1^\F(C_\ast)$ and $\|s\|^{(s(1))} =1$. 
  Hence, $G_1^\F (C_*,s) = 1$.  Since $s$ is the only chain extension
  of $C_*$ in degree $1$ over $\F$, we conclude~$G_1^\F(C_*)=1$.
\end{example}

\begin{example}\label{exa:mapcylgnarl}
  Let $d \in \N_{>0}$. 
  We consider the $\Z$-chain complex~$C_*$ given by
  \[
  \begin{tikzcd}
    \Z
    \ar{r}{(d,1)}
    & 
    \Z \oplus \Z
    \ar{r}{0}
    & \Z,
  \end{tikzcd}
  \]
  concentrated in degrees~$2$, $1$,~$0$ with the marked bases given
  by~$1 \in \Z$ and $(e_1,e_2) := ((1,0), (0,1))$ of~$\Z\oplus
  \Z$. This is related to a cellular chain complex of the mapping
  cylinder of a $d$-fold covering~$S^1 \to S^1$.

  We have~$G_1^\F(C_*) = d$ for every subfield~$\F$ of $\R$, which can be
  calculated as follows: 
  Because $\partial_1=0$, there is again only one chain extension $s$
  of $C_\ast$ in degree $1$ over~$\F$, namely the canonical projection
  map
  \[ s \colon C_1 \otimes_\Z \F = \ker (\partial_1\otimes_\Z \id_\F)
  \to H_1(C_*;\F).
  \]
  Because $s$ is a chain extension, we have~$s((d,1)) = 0$ and so
  \[ s(e_2) = s( e_2 - (d,1) ) = -d \cdot s(e_1);
  \]
  in particular, surjectivity of~$s$ shows that $s(e_1) \neq 0$. 
  If $B \in B_1^\F(C_*)$ is a cointegral basis of~$H_1(C_*;\F)$,
  then the above calculation and the integrality condition lead to
  \[ \|s\|^B
  = \max \bigl( |s(e_1)|_1^B, |s(e_2)|_1^B \bigr)
  = d \cdot |s(e_1)|_1^B
  \geq d \cdot 1
  = d.
  \]
  Hence, $G_1^\F(C_*,s)\geq d$.
  Conversely, $B := (s(e_1))$ lies in~$B_1^\F(C_*)$ and
  satisfies~$\|s\|^B = d$.
  Thus, $G_1^\F(C_*) = G_1^\F(C_*,s) = d$.
\end{example}

\begin{example}\label{exa:mapcylgnarlcovering}
  The above Example~\ref{exa:mapcylgnarl} can be modified so that we
  obtain every rational number as the gnarledness of a finite marked
  free $\Z$-chain complex.
  
  We do this by looking at the cellular chain complex of finite
  coverings of the CW-structure behind Example~\ref{exa:mapcylgnarl}:
  Let $m,d\in \N_{>0}$ and let the $\Z$-chain complex $C_*\langle
  m\rangle$ be given by
  \[
  \begin{tikzcd}[row sep=0.1em]
    \Z^m
    \ar{r}{\partial_2}
    & 
    \Z^m \oplus \Z^m
    \ar{r}{\partial_1}
    & \Z^m\\
    e_k 
    \arrow[r, maps to]
    &\Bigl(\sum_{i=0}^{d-1} e_{k+i}, e_k\Bigr) &\\
    & (e_k,0) 
    \arrow[r, maps to]
    & -e_k + e_{k+1}\\
    & (0,e_k) 
    \arrow[r, maps to]
    & +e_k - e_{k+d}
  \end{tikzcd}
  \]
  concentrated in degrees $2$, $1$, $0$ and marked with the standard basis in every degree. 
  Indices shall be considered modulo $m$.
  Then $G_1^\F\bigr(C_*\langle m\rangle\bigl) = \frac{d}{m}$:

  It is straightforward to check that $H_1\bigl(C_*\langle
  m\rangle;\F\bigr)$ is one-dimensional and that the integral part
  $H_1\bigl(C_*\langle m\rangle \bigr)_\F$ is generated by the
  homology class
  \[
  b:= \Biggl[\biggl(\sum_{i=1}^{m} e_{i},0\biggr)\Biggr].
  \]
  Similarly to Example~\ref{exa:mapcylgnarl}, we have 
  \[ \Biggl[\biggl(0,\sum_{i=1}^{m} e_{i}\biggr)\Biggr]
  = \Biggl[-\biggl( \sum_{i=1}^m\sum_{j=1}^{d-1} e_{j+i},0\biggr)\Biggr]
  = \Biggl[-d\cdot\biggl( \sum_{i=1}^{m} e_{i},0\biggr)\Biggr]. 
  \]
  Let $B$ be a cointegral basis of $H_1\bigl(C_*\langle m\rangle;\F\bigr)$. 
  Then $B = \bigl(\frac{1}{\lambda}\cdot b\bigr)$ for some non-zero integer~$\lambda$. 
  If $s\colon C_1\langle m\rangle\otimes_\Z\F \to H_1\bigl(C_*\langle
  m\rangle;\F\bigr)$ is a chain extension of~$C_*\langle m\rangle$ in
  degree $1$ over $\F$, we obtain
  \begin{align*} 
    \sum_{i=1}^m \bigl|s(0,e_i)\bigr|_1^B &\geq \Biggl|s\biggl(0,\sum_{i=1}^m e_i\biggr)\Biggr|_1^B
    = \Biggl|\Biggl[\biggl(0,\sum_{i=1}^m e_i\biggr)\Biggr]\Biggr|_1^B
    = \Biggl|-d\cdot\Biggl[\biggl(\sum_{i=1}^m e_i,0\biggr)\Biggr]\Biggr|_1^B
    \\
    &
    = d \cdot| b|^B = d\cdot |\lambda|.
  \end{align*}
  Hence, we have $\|s\|^B \geq \frac{d\cdot |\lambda|}{m}$.
  Taking the infimum over all chain extensions and cointegral bases gives $G_1^\F(C_*\langle m\rangle) \geq \frac{d}{m}$.

  On the other hand, we consider the $\F$-linear map given by 
  \begin{align*}
    s\colon C_1\langle m\rangle \otimes_\Z \F &\to H_1(C_*\langle m\rangle;\F)\\
     (e_k,0) &\mapsto \frac{1}{m} \cdot b\\
     (0,e_k) &\mapsto -\frac{d}{m} \cdot b.
  \end{align*}
  Then $s$ is a chain extension of $C_*\langle m\rangle$ in degree $1$
  over $\F$ and for the cointegral basis $B=(b)$ we compute
  $G_1^\F\bigl(C_*\langle m\rangle\bigr) \leq \|s\|^B = \frac{d}{m}$.

  This inequality is also a consequence of monotonicity of gnarledness
  (Proposition~\ref{prop:gnarlmono}): Since $C_*\langle m\rangle$ is
  related to the cellular complex of the $m$-fold covering of the
  mapping cylinder of Example~\ref{exa:mapcylgnarl}, transfer gives us
  the chain map
  \begin{align*}
    f_* : C_*\langle m\rangle\otimes_\Z \F &\to C_*\otimes_\Z\F\\
    \sigma &\mapsto \frac{1}{m}\cdot p(\sigma),
  \end{align*}
  where $C_*$ denotes the chain complex of
  Example~\ref{exa:mapcylgnarl}, $p$ is the covering map and $\sigma$
  is a cell of the total space.  By monotonicity and
  Example~\ref{exa:mapcylgnarl}, we obtain
  \[ G_1^\F\bigl(C_*\langle m\rangle\bigr) \leq \frac{1}{m} \cdot G_1^\F(C_*) = \frac{d}{m}.
  \]
\end{example}

%%%%%%%%%%%%%%%%%
\subsection{Comparison with the lifting volume}\label{subsec:liftgnarl}

Up to a Betti number factor, gnarledness coincides with the
lifting volume. Hence, gnarledness provides good estimates
for liftings and integral fillings of cohomology classes.

\begin{prop}\label{prop:liftgnarl}
  Let $C_*$ be a marked free $\Z$-chain complex, 
  let $k \in \N$, and let $\F$ be a subfield of~$\R$. 
  Let $H_k(C_*;\Z)$ be finitely generated. % only needed for the second part
  Then:
  \[ L_k^\F(C_*)
  \leq \frac12 \cdot G_k^\F(C_*)
  \leq \dim_\F H_k(C_*;\F) \cdot L_k^\F(C_*).
  \]
\end{prop}
\begin{proof}
  The first estimate is the chain extension estimate
  (Corollary~\ref{cor:liftchainext}).

  For the second estimate, we construct chain extensions
  from efficient liftings:
  More precisely, we show the following for each~$\varepsilon \in \R_{>0}$:
  \begin{enumerate}
  \item\label{i:liftgnarl:exdualco}
    There exists
    a cointegral basis~$B$ of~$H_k(C_*;\F)$ 
    such that the dual family~$(\varphi_\beta)_{\beta \in B}$ in~$H^k(C_*;\F)$ 
    admits representing cocycles~$(f_\beta)_{\beta \in B}$ with
    \[ |f_\beta|_\infty \leq 2 \cdot \bigl(L_k^\F(C_*) + \varepsilon\bigr).
    \]
  \item\label{i:liftgnarl:chainext}
    The associated map 
    \begin{align*}
      s \colon
      C_k \otimes_\Z \F
      & \longrightarrow H_k(C_*;\F)
      \\
      x \otimes \lambda
      & \longmapsto
      \lambda \cdot 
      \sum_{\beta \in B} f_\beta(x) \cdot \beta
    \end{align*}
    is a chain extension of~$C_*$ over~$\F$ in degree~$k$ 
  \item\label{i:liftgnarl:chainextnorm}
    and~$\|s\|^B \leq 2 \cdot \dim_\F H_k(C_*;\F) \cdot \bigl(L_k^\F(C_*) + \varepsilon\bigr)$.
  \end{enumerate}
  Taking~$\varepsilon \to 0$ then proves the claimed estimate.

  \emph{Ad~\ref{i:liftgnarl:exdualco}.} 
  Let $B$
  be a cointegral basis of~$H_k(C_*;\F)$ obtained from
  a $\Z$-basis of~$H_k(C_*)_\F$.
  Let $(\psi_\beta)_{\beta \in B}$ in~$H^k(C_*;\F)$ be dual to~$B$.
  We apply the lifting volume to suitable reductions of this family;
  however, as this family is dual to~$B$ its reduction in~$H^k(C_*;\F/\Z)$
  is trivial. Therefore, to get meaningful classes, for~$\beta \in B$, 
  let $\overline \psi_\beta := \pi(1/2 \cdot \psi_\beta) \in H^k(C_*;\F/\Z)$.
  By definition of the lifting volume, there exist $\pi$-lifts~$\varphi_\beta
  \in H^k(C_*;\F)$ of~$\overline \psi_\beta$ and cocycles~$f_\beta
  \in \Hom_\Z(C_k;\F)$ representing~$\varphi_\beta$ with
  \[ |f_\beta|_\infty \leq L_k^\F(C_*) + \varepsilon.
  \]
  In general, the family~$(2 \cdot \varphi_\beta)_{\beta \in B}$ will not be dual
  to~$B$, but we still have almost-duality of the following form:
  For all~$\beta,\gamma \in B$, we have
  \[ \langle 2\cdot \varphi_\beta, \gamma \rangle
  \in
  2 \cdot \Bigl(
  \Z + \Bigl\langle \frac12\cdot \psi_\beta, \gamma\Bigr\rangle
  \Bigr)
  = 
  2 \cdot \Bigl(
  \Z + \frac12 \cdot \langle \psi_\beta, \gamma\rangle
  \Bigr)
  = 
  \begin{cases}
    2 \cdot \Z & \text{if $\beta\neq\gamma$}
    \\
    2 \cdot \Z + 1 & \text{if $\beta=\gamma$},
  \end{cases}
  \]
  because $\pi(\varphi_\beta) = \pi(1/2 \cdot \psi_\beta)$
  and the family~$(\psi_\beta)_{\beta \in B}$ is dual to~$B$. 
  
  Applying the algebraic Lemma~\ref{lem:exdualco} below shows
  that there exists a cointegral basis~$B'$ of~$H_k(C_*;\F)$
  dual to the family~$(2 \cdot \varphi_\beta)_{\beta \in B}$.
  Moreover, $2 \cdot f_\beta$ is a cocycle
  representing~$2 \cdot \varphi_\beta$ with
  \[ |2 \cdot f_\beta|_\infty \leq 2 \cdot \bigl(L_k^\F(C_*) + \varepsilon\bigr).
  \]
  Renaming the family and bases thus proves the first claim. 
  In the following, we assume that $B$, $(\varphi_\beta)_{\beta \in B}$,
  and~$(f_\beta)_{\beta \in B}$ are chosen as in Claim~\ref{i:liftgnarl:exdualco}. 
  
  \emph{Ad~\ref{i:liftgnarl:chainext}.}
  Let $\Sigma_k$ be the marked basis of~$C_k$
  and let $c = \sum_{\sigma \in \Sigma_k} \sigma \otimes \lambda_\sigma \in C_k \otimes_\Z \F$
  be a cycle. By construction of~$s$ as in Claim~\ref{i:liftgnarl:chainext}, we have 
  \begin{align*}
    s(c)
    & = \sum_{\sigma \in \Sigma_k} \lambda_\sigma \cdot \sum_{\beta \in B} f_\beta(\sigma) \cdot \beta
    = \sum_{\beta \in B}
    \biggl( \sum_{\sigma \in \Sigma_k} \lambda_\sigma \cdot f_\beta(\sigma) \biggr)
    \cdot \beta
    \\
    & = \sum_{\beta \in B}
    \bigl\langle [f_\beta], [c] \bigr\rangle
    \cdot \beta
    = \sum_{\beta \in B}
    \bigl\langle \varphi_\beta, [c] \bigr\rangle
    \cdot \beta.
  \end{align*}
  Because $(\varphi_\beta)_{\beta \in B}$ is dual to~$B$, this
  simplifies to~$s(c) = [c]$.
  Thus, $s$ is a chain extension.

  \emph{Ad~\ref{i:liftgnarl:chainextnorm}.}
  By definition, 
  $
    \|s \|^B
    = \max_{\sigma \in \Sigma_k} \bigl|s(\sigma \otimes 1)\bigr|_1^B.
  $ 
  For each~$\sigma \in \Sigma_k$, we obtain
  \begin{align*}
    \bigl|s(\sigma \otimes 1)\bigr|_1^B
    & = \biggl| \sum_{\beta \in B} f_\beta(\sigma) \cdot \beta \biggr|_1^B
    = \sum_{\beta \in B} \bigl|f_\beta(\sigma)\bigr|
    \leq \sum_{\beta \in B} |f_\beta|_\infty
    \\
    & \leq \# B \cdot 2 \cdot \bigl(L_k^\F(C_*) + \varepsilon\bigr)
    = 2 \cdot \dim_\F H_k(C_*;\F) \cdot \bigl(L_k^\F(C_*) + \varepsilon\bigr).
    \qedhere
  \end{align*}
\end{proof}

In particular, if $\dim_\F H_k(C_*;\F)=1$,
then $2 \cdot L_k^\F(C_*) = G_k^\F(C_*)$.  

\begin{lemma}\label{lem:exdualco}
  Let $n \in \N_{>0}$ and let $(\varphi_j \colon \F^n \to \F)_{j \in \{1,\dots,n\}}$
  be $\Z$-linear maps with the following property:
  \[ \fa{j,k \in \{1,\dots,n\}}
  \varphi_j(e_k) \in
  \begin{cases}
    2 \cdot \Z & \text{if $j \neq k$}
    \\
    2 \cdot \Z + 1 & \text{if $j = k$}.
  \end{cases}
  \]
  Then, there exists a cointegral basis~$B$ of~$\F^n$ such that
  $(\varphi_j)_{j \in \{1,\dots,n\}}$ is dual to~$B$.
  Here, \emph{cointegral} means that $\Z^n \subset \Z B$.
\end{lemma}
\begin{proof}
  We consider the matrix
  \[ A :=
  \bigl( \varphi_j(e_k) \bigr)_{j,k \in \{1,\dots,n\}}.
  \]
  In view of our hypothesis, this matrix has integer coefficients,
  whose diagonal entries are odd and whose 
  off-diagonal entries are even. By the Leibniz
  expansion formula, $\det A$ is an odd integer. In particular,
  $\det A \neq 0$. Therefore, there exists a
  matrix~$\overline A \in \mat nn \Q$ that is inverse to~$A$.
  Let $B$ be the family of columns of~$\overline A$. By invertibility, 
  $B$ is an $\F$-basis of~$\F^n$ and $(\varphi_j)_{j \in \{1,\dots, n\}}$
  is dual to~$B$.
  
  It remains to show that $B$ is cointegral. Integrality of~$A$ and
  $\overline A \cdot A = I_n$ witness that each~$e_k$ lies in~$\Z
  B$. So, $B$ is cointegral.
\end{proof}

\begin{cor}[characterisation of vanishing]\label{rem:char0}
  Let $C_*$ be a marked free $\Z$-chain complex, 
  let $k \in \N$, and let $\F$ be a subfield of~$\R$.
  Let $H_k(C_*)$ be finitely generated. 
  Then the following are equivalent:
  \begin{enumerate}
  \item\label{char0.1}
    We have $G_k^\F(C_*) = 0$.
  \item\label{char0.2}
    We have $L_k^\F(C_*) = 0$.
  \item\label{char0.3}
    We have~$H_k(C_*;\F) \cong_\F 0$.
  \end{enumerate}
\end{cor}
\begin{proof}
  The equivalence $\ref{char0.1}. \Longleftrightarrow \ref{char0.2}.$
  follows from the sandwich estimate in Proposition~\ref{prop:liftgnarl}.

  The implication~$\ref{char0.3}. \Longrightarrow \ref{char0.2}.$ is
  immediate from the definition.

  For the implication~$\ref{char0.2}. \Longrightarrow \ref{char0.3}.$,
  we proceed as follows: Let $H_k(C_*;\F) \not \cong_\F 0$.
  We show that $L_k^\F(C_*) > 0$. 
  Because $H_k(C_*)$
  is finitely generated and $H_k(C_*;\F) \not\cong_\F 0$, 
  the canonical 
  map~$\Hom_\Z(H_k(C_*),\Z) \to \Hom_\Z(H_k(C_*),\F)$
  is not surjective. By the universal coefficient theorem
  (and divisibility of~$\F$), 
  therefore, the change-of-coefficients map~$H^k(C_*;\Z) \to H^k(C_*;\F)$
  is not surjective. 
  In view of the long exact coefficient sequence, 
  the reduction map~$\pi \colon H^k(C_*;\F) \to H^k(C_*;\F/\Z)$ is 
  non-trivial.

  Let $\varphi \in \im \pi$ be non-trivial. Then, by the universal
  coefficient theorem (and divisibility of~$\F/\Z$), there exists a
  class~$\alpha \in H_k(C_*)$ with~$\langle \varphi, \alpha\rangle
  \neq 0$ in~$\F/\Z$. Rescaling $\varphi$, we may assume that $\langle
  \varphi, \alpha\rangle$ is represented by~$1/2$.

  Let $c \in C_k$ be a cycle representing~$\alpha$.
  We show that $L_k^\F(C_*,\varphi) \geq 1/(2 \cdot |c|_1)$: 
  Every cocycle~$f \in \Hom_\Z (C_k,\F)$ representing a $\pi$-lift of~$\varphi$
  satisfies $\langle [f] , \alpha \rangle \in \Z + 1/2$ and so
  \[ |f|_\infty \cdot |c|_1
  \geq \bigl|f(c)\bigr|
  =  \bigl| \langle [f],\alpha \rangle \bigr|
  \geq \frac12.
  \]
  Hence, $L_k^\F(C_*) \geq L_k^\F(C_*,\varphi) \geq 1/(2 \cdot |c|_1) > 0$.
\end{proof}

%%%%%%%%%%%%%%%%%%%%%%%%%%%%%%%%%%%%%%%%%%%%%%%%%%
\subsection{Universality}\label{subsec:gnarluniversal}

Gnarledness is universal among size invariants in
the sense of Section~\ref{sec:size}.

\begin{prop}[universality of gnarledness]\label{prop:universality}
  Let $\F$  be a subfield of~$\R$ and let $k \in \N$.
  Then, $G_k^\F$ is universal among
  all size invariants over~$\F$ in degree~$k$:
  \begin{enumerate}
  \item Gnarledness~$G_k^\F$ is a size invariant over~$\F$ in degree~$k$.
  \item
    If $I \colon \mCh \to [0,\infty]$
    is a size invariant over~$\F$ in degree~$k$ and $C_*$ is a marked free
    $\Z$-chain complex, then
    \[ I(C_*) \leq G_k^\F(C_*).
    \]
    \end{enumerate}
\end{prop}
\begin{proof}
  \emph{Ad~1.} 
  It is clear from the definition that gnarledness satisfies the
  normalisation property. We show monotonicity
  below (Proposition~\ref{prop:gnarlmono}).

  \emph{Ad~2.}
  This is immediate from the chain extension estimate
  for size invariants (Proposition~\ref{prop:sizechainext}) and
  the definition of gnarledness.
\end{proof}

\begin{prop}[monotonicity]\label{prop:gnarlmono}
  Let $C_*$ and $D_*$ be marked free $\Z$-chain complexes,
  let $k \in \N$, and let $\F$ be a subfield of~$\R$.  Let $f_* \colon
  C_* \otimes_\Z \F \to D_* \otimes_\Z \F$
  be an $\F$-chain map that induces an $\F$-isomorphism 
  $H_k(f_*) \colon H_k(C_*;\F) \to
  H_k(D_*;\F)$
  and that satisfies the integrality condition~$H_k(f_*)(H_k(C_*)_\F)
  \subset H_k(D_*)_\F$. Then, we have
  \[ G_k^\F(C_*) \leq \|f_k\| \cdot G_k^\F(D_*),
  \]
  where $\|f_k\|$ refers to the operator norm taken with respect
  to the $\ell^1$-norms associated with the marked bases on~$C_k$
  and~$D_k$, respectively.
\end{prop}
\begin{proof}
  This is a straightforward estimate, based on the observation
  in Lemma~\ref{lem:chainextheq}: 
  Let $s \colon D_k \otimes_\Z \F \to H_k(D_*;\F)$ be a chain
  extension of~$D_*$ in degree~$k$ over~$\F$ and let $B$ be
  a cointegral basis of~$H_k(D_*;\F)$.
  Then 
  \[ \widetilde s := H_k(f_*)^{-1} \circ s \circ f_k 
  \colon C_k \otimes_\Z \F \to H_k(C;\F)
  \]
  is a chain extension of~$C_*$ in degree~$k$ over~$\F$
  (Lemma~\ref{lem:chainextheq}). We consider the family
  $\widetilde B := H_k(f_*;\F)^{-1}(B).
  $ 
  Then $\widetilde B$ is an $\F$-basis of~$H_k(C_*;\F)$.
  The integrality condition and cointegrality of~$B$
  ensure that $\widetilde B$ is also cointegral:
  Indeed,
  \begin{align*}
    H_k(C_*)_\F
  & = H_k(f_*)^{-1} \circ H_k(f_*)\bigl(H_k(C_*)_\F\bigr)
    \subset H_k(f_*)^{-1} \bigl( H_k(D_*)_\F \bigr)
  \\
  & \subset H_k(f_*)^{-1} (\Z B)
  \subset \Z \bigl(H_k(f_*)^{-1}(B)\bigr)
  \\
  & = \Z \widetilde B.
  \end{align*}

  By construction of~$\widetilde B$, we have~$\|H_k(f_*)^{-1}\|^{B, \widetilde B} \leq 1$.
  Therefore, we obtain
  \begin{align*}
    G_k^\F(C_*)
  & \leq \|\widetilde s\|^{\widetilde B}
  = \bigl\| H_k(f_*)^{-1} \circ s \circ f_k
  \bigr\|^{\widetilde B}
  \\
  & \leq \|H_k(f_*)^{-1}\|^{B, \widetilde B} \cdot \|s\|^B \cdot \|f_k\|
  \leq 1 \cdot \|s\|^B \cdot \|f_k\|.
  \end{align*}
  Taking the infimum over~$B$ and~$s$ gives the claimed estimate.
\end{proof}

As a consequence, we obtain via the sandwich estimate
(Proposition~\ref{prop:liftgnarl}) that the cohomological lifting
volume is also (almost) universal.
Universality can also be used to compare gnarledness
over different coefficients (Corollary~\ref{cor:gnarlcoeffirrelevant}).

%%%%%%%%%%%%%%%%%
\subsection{Comparison with geometric gnarledness}\label{subsec:gnarlgeo}

For the comparison with the geometric notion of gnarledness,
we recall the original definition of gnarledness
of simplicial complexes~\cite{CDMW}:

\begin{defi}[gnarledness of simplicial chain complexes]
  Let $X$ be a finite simplicial complex and let $k \in \N$.
  \begin{itemize}
  \item A \emph{$k$-spanning tree} of~$X$ is a $k$-dimensional
    subcomplex~$T$ of~$X$ with the following properties:
    \begin{itemize}
    \item We have $T_{k-1} = X_{k-1}$.
    \item The map~$H_{k-1}(T;\Q) \to H_{k-1}(X;\Q)$
      induced by the inclusion is a $\Q$-isomorphism.
    \item We have~$H_k(T;\Q) \cong_\Q 0$.
    \end{itemize}
  \item Let $T$ be a $k$-spanning tree of~$X$.
    We write~$B_k^\Q(T)$ for the set of all $\Q$-bases~$B$ of~$H_k(X,T;\Q)$
    that satisfy for all~$\sigma \in X_k$ that 
    $[\sigma] \in \Z B.
    $ 
    For~$B \in B_k^\Q(T)$, we set
    \[ G_k(X,T,B) := \max_{\sigma \in X_k} \bigl| [\sigma] \bigr|_1^B.
    \]
  \item
    The \emph{gnarledness of~$X$} in degree~$k$ is defined as
    \[ G_k(X) := \min \bigl\{ G_k(X,T,B)
    \bigm| \text{$T$ is a $k$-spanning tree of~$X$, $B \in B_k^\Q(T)$}
    \bigr\}.
    \]
  \end{itemize}
\end{defi}

In the definition of~$G_k(X,T,B)$, because
of~$T_{k-1} = X_{k-1}$ indeed every simplex~$\sigma \in X_k$
represents a relative cycle for~$(X,T)$ and so represents a class
in~$H_k^\Delta(X,T;\Q)$.

\begin{prop}
  Let $X$ be a finite simplicial complex and let $k \in \N$. Then:
  \[ G_k^\Q \bigl( C_*^\Delta(X) \bigr)
  \leq G_k(X).
  \]
\end{prop}
\begin{proof}
  Let $T$ be a $k$-spanning tree of~$X$ and let $B \in B_k^\Q(T)$.
  We explain how to obtain a chain extension for~$C_*^\Delta(X)$
  from~$T$, which will then witness the claimed estimate.

  Because $T$ is a $k$-spanning tree of~$X$, the long exact
  pair sequence for~$(X,T)$ shows that the map
  \[ H_k^\Delta(X;\Q) \to H_k^\Delta(X,T;\Q)
  \]
  induced by the inclusion is a $\Q$-isomorphism. Let $t \colon
  H_k^\Delta(X,T;\Q) \to H_k^\Delta(X;\Q)$ be the inverse isomorphism
  and let~$r\colon C_k^\Delta(X;\Q)
  \to H_k^\Delta(X,T;\Q)$ 
  be the homomorphism induced by the map~$X_k \to H_k^\Delta(X,T;\Q)$
  given by~$\sigma \to [\sigma]$. We consider
  \[ s := t \circ r \colon C_k^\Delta(X;\Q) \to H_k^\Delta(X;\Q)
  \qand
  \widetilde B := t(B).
  \]
  We proceed in the following steps:
  \begin{enumerate}
  \item The map~$s$ is a chain extension of~$C_*^\Delta(X)$
    in degree~$k$ over~$\Q$, because:
    Let $c \in C_k^\Delta(X;\Q)$ be a cycle.
    Then $t([c]) = [c]$ (because $c$ is an actual cycle on~$X$)
    and so
    \[ s(c)
    = t \circ r (c)
    = t \bigl([c]\bigr)
    = [c].
    \]
  \item We have~$\widetilde B \in B_k^\Q(C_*^\Delta(X))$, because:
    By construction, $\widetilde B$ is a $\Q$-basis of~$H_k^\Delta(X;\Q)$. 
    To show the cointegrality condition, let $\alpha \in H_k^\Delta(X;\Z)_\Q$.
    Then $\alpha$ is represented by a cycle~$c \in C_k^\Delta(X) = \Z X_k$.
    Because $s$ is a chain extension, we obtain in~$H_k^\Delta(X;\Q)$:
    \[ \alpha
    = [c]
    = s(c)
    = t\bigl([c]\bigr)
    \in t(\Z B)
    \subset \Z t(B)
    = \Z \widetilde B
    \]
    Hence, $H_k^\Delta(X;\Z)_\Q \subset \Z \widetilde B$.
  \item We have
    \begin{align*}
      G_k^\Q\bigl(C_*^\Delta(X)\bigr)
    & \leq G_k^\Q\bigl(C_*^\Delta(X),s\bigr)
    \leq \|s\|^{\widetilde B}
    = \max_{\sigma \in X_k} \bigl|s(\sigma)\bigr|_1^{\widetilde B}
    = \max_{\sigma \in X_k} \bigl| [\sigma]\bigr|_1^B
    \\
    &
    = G_k(X,T,B).
    \end{align*}
  \end{enumerate}
  Finally, taking the minimum over all $k$-spanning trees~$T$
  of~$X$ and all~$B \in B_k^\Q(T)$, shows that 
  $G_k^\Q\bigl(C_*^\Delta(X)\bigr) \leq G_k(X)$.
\end{proof}

\begin{rem}
  Let $X$ be a finite simplicial complex, let $k \in \N$, and
  let $s$ be a chain extension of~$C_*^\Delta(X)$ in degree~$k$
  over~$\Q$.
  We may then define a $k$-dimensional subcomplex~$T$ of~$X$ by setting 
  \begin{align*}
    T_{k-1} & := X_{k-1}
    \\
    T_k & := \bigl\{ \sigma \in X_k \bigm| s(\sigma) = 0 \in H_k^\Delta(X;\Q)
    \bigr\}.
  \end{align*}
  By construction, the canonical map~$H_{k-1}^\Delta(T;\Q) \to H_{k-1}^\Delta(X;\Q)$
  is surjective and $H_k^\Delta(T;\Q) \to H_k^\Delta(X;\Q)$ is zero. However,
  in general, we cannot expect $T$ to be an actual $k$-spanning tree of~$X$.

  An example where $G_k^\Q(C_*^\Delta(X))$ and $G_k(X)$ do not coincide
  is given below (Example~\ref{exa:longcircle}), but
  for the purpose of cohomological filling estimates, we may work
  with the algebraic version~$G_k^\Q(C_*^\Delta(X))$ instead of~$G_k(X)$.

  Moreover, $G_k^\Q(C_*^\Delta(X))$ satisfies monotonicity
  (Proposition~\ref{prop:gnarlmono}); this applies, in particular,
  to simplicial maps that induce $\Q$-homology isomorphisms.
  In contrast, for the gnarledness of simplicial complexes,
  it is not clear whether such a monotonicity property holds.
\end{rem}

\begin{example}\label{exa:longcircle}
  Let $m \in \N_{\geq 3}$ and let $X$ be the simplicial complex
  consisting of $m$~edges chained together as a circle of length~$m$.
  Let $C_* := C_*^\Delta(X)$ be the associated simplicial chain complex.
  Then $C_0 \cong_\Z \Z^m$ and $C_1 \cong_\Z \Z^m$, where the
  isomorphisms map the vertices and the (positively oriented) edges,
  respectively, to the standard basis of~$\Z^m$.

  On the one hand, by definition, $G_1(X) \in \N$ and clearly
  $G_1(X) \neq 0$. In particular, $G_1(X) \geq 1$. 

  On the other hand, 
  \begin{align*}
    s \colon C_1 \otimes_\Z \Q
    & \longrightarrow H_1^\Delta(X;\Q)
    \\
    e_j \otimes 1
    & \longmapsto \frac1m \cdot [e_1 + \dots + e_m]
  \end{align*}
  is a chain extension of~$C_*$ in degree~$1$ over~$\Q$
  and $B := ([e_1+ \dots e_m])$ is a cointegral basis of~$H_1^\Delta(X;\Q)$.
  Therefore,
  \[ G_1^\Q (C_*)
  \leq \|s\|^B
  = \frac1m < 1 \leq G_1(X).
  \]
  Viewing this simplicial circle as an $m$-fold covering of
  the CW-structure of~$S^1$ with a single $1$-cell, this example
  is related to Example~\ref{exa:mapcylgnarlcovering}.
\end{example}

Geometric gnarledness is algorithmically computable (though it is not
clear whether this can be done efficiently). In contrast,
we do not know whether this is true for the algebraic gnarledness
or the cohomological lifting volume of marked $\Z$-chain complexes of finite type.

\begin{prop}\label{prop:geognarlalgo}
  There exists an algorithm that
  \begin{itemize}
  \item given a finite simplicial complex~$X$ and~$k \in \N$,
  \item computes~$G_k(X)$.
  \end{itemize}
\end{prop}
\begin{proof}
  It suffices to prove the following two steps:
  \begin{enumerate}
  \item\label{geognarlalgo.i}
    There exists an algorithm that
    \begin{itemize}
    \item given a finite simplicial complex~$X$ and $k \in \N$,
    \item computes the set of all $k$-spanning trees of~$X$.
    \end{itemize}
  \item\label{geognarlalgo.ii}
    There exists an algorithm that
    \begin{itemize}
    \item given a finite simplicial complex~$X$, a $k\in \N$,
      a $k$-spanning tree~$T$ of~$X$, and $K \in \N$
    \item computes the (finite) set of all bases~$B$ in~$B_k^\Q(T)$
      with
      \[ \max_{\sigma \in X_k} \bigl|[\sigma]\bigr|_1^B \leq K.
      \]
    \end{itemize}
  \end{enumerate}
  Indeed, using these algorithms, we can compute~$G_k(X)$
  as the minimal number~$K$ such that there exists a $k$-spanning tree~$T$
  for which the set computed by the second algorithm is non-empty.

  \emph{Ad~\ref{geognarlalgo.i}.} 
  Because $X$ is finite, we can compute the finite set of all
  subcomplexes of~$X$. For each subcomplex~$T$ of~$X$, we can
  check algorithmically whether $T$ is a $k$-spanning tree of~$X$,
  as the simplicial homology functors over~$\Q$ on finite simplicial complexes 
  are algorithmically computable.

  \emph{Ad~\ref{geognarlalgo.ii}.}
  Given a finite simplicial complex~$X$, a $k \in \N$, a $k$-spanning tree~$T$
  of~$X$, and $K \in \N$, we know, by the universal coefficient theorem,
  that $\{[\sigma] \mid \sigma \in X_k\}$ is a $\Q$-generating set of~$H_k(X,T;\Q)$
  and thus contains a $\Q$-basis~$B$ of~$H_k(X,T;\Q)$.

  We apply Lemma~\ref{lem:basisalgo} below to
  the $\Q$-isomorphism~$f \colon \Q^{\dim_\Q H_k(X,T;\Q)} \to H_k(X,T;\Q)$
  induced by~$B$ and the subset~$\{f^{-1}([\sigma]) \mid \sigma \in X_k\}$
  (which is algorithmically computable). Finally, we translate the result
  back via~$f$. 
\end{proof}

\begin{lemma}\label{lem:basisalgo}
  There is an algorithm that
  \begin{itemize}
  \item given $d\in \Q$, a $\Q$-generating set~$\Sigma \subset \Q^d$
    with $\{e_1, \dots, e_d\}\subset \Sigma$, and $K\in \N$
  \item computes the set
    \[ \bigl\{ B \bigm|
    \text{$B$ is a basis of~$\Q^d$
      with $\Sigma \subset \Z B$
      and $\max_{\sigma \in \Sigma} |\sigma|_1^B \leq K$}
    \bigr\}.
    \]
  \end{itemize}
\end{lemma}
\begin{proof}
  For~$d\in \Q$ and a $\Q$-generating set~$\Sigma \subset \Q^d$
  with $\{e_1, \dots, e_d\}\subset \Sigma$, we have
  \begin{align*}
    & \; \{ B \mid 
    \text{$B$ is a basis of~$\Q^d$
      with $\Sigma \subset \Z B$}
    \}
    \\
    = & \;
    \biggl\{ (A_{*1}, \dots, A_{*d}) \biggm|
    A \in \GL_d(\Q),\
    \fa{\sigma \in \Sigma} \exi{\lambda_1, \dots, \lambda_d \in \Z}
    \sigma = \sum_{j=1}^d \lambda_j \cdot A_{*j}
    \biggr\}
    \\
    = & \;
    \biggl\{ (A_{*1}, \dots, A_{*d}) \biggm|
    A \in \GL_d(\Q),\
    A^{-1} \in \mat dd\Z,\ 
    \fa{\sigma \in \Sigma\setminus\{e_1,\dots,e_d\}} \exi{\lambda_1, \dots, \lambda_d \in \Z}
    \sigma = \sum_{j=1}^d \lambda_j \cdot A_{*j}
    \biggr\}
    \\
    = & \;
    \biggl\{ (A^{-1}_{*1}, \dots, A^{-1}_{*d}) \biggm|
    A \in \mat dd \Z,\ 
    \det A \neq 0,\ 
    \fa{\sigma \in \Sigma\setminus\{e_1,\dots,e_d\}} \exi{\lambda_1, \dots, \lambda_d \in \Z}
    \sigma = \sum_{j=1}^d \lambda_j \cdot A^{-1}_{*j}
    \biggr\}.
  \end{align*}
  Hence, given~$K \in \N$, we obtain
  \begin{align*}
    & \;
    \bigl\{ B \bigm|
    \text{$B$ is a basis of~$\Q^d$
      with $\Sigma \subset \Z B$
      and $\max_{\sigma \in \Sigma} |\sigma|_1^B \leq K$}
    \bigr\}
    \\
    = & \;
    \biggl\{ (A^{-1}_{*1}, \dots, A^{-1}_{*d}) \biggm|
    A \in \mat dd \Z,\ 
    \det A \neq 0,\
    \fa{j \in \{1,\dots,d\}} |A_{*j}|_1 \leq K,\
    \\
    & \; \phantom{\bigl\{ (A^{-1}_{*1}, \dots, A^{-1}_{*d}) \bigm|\;}
    \fa{\sigma \in \Sigma\setminus\{e_1,\dots,e_d\}} \exi{\lambda_1, \dots, \lambda_d \in \{-K,\dots,K\}}
    \sigma = \sum_{j=1}^d \lambda_j \cdot A^{-1}_{*j} \land \sum_{j=1}^d |\lambda_j| \leq K
    \biggr\},   
  \end{align*}
  which is algorithmically computable due to the a~priori bound on the~$A_{*j}$.
\end{proof}

%%%%%%%%%%%%%%%%%
\subsection{Comparison with simplicial volume}\label{subsec:gnarlsv}

We reformulate the duality principle for simplicial volume~\cite{vbc}
in terms of gnarledness of the singular chain complex:

\begin{prop}\label{prop:gnarlsingular}
  Let $C_*$ be a marked free $\Z$-chain complex and let $n \in \N$
  with~$H_n(C_*;\R) \cong_\R \R$. Let $\mu_\Z \in H_n(C_*)$
  be an element such that the induced class~$\mu_\R \in H_n(C_*;\R)$
  is a $\Z$-generator of~$H_n(C_*)_\R$. Then
  \[ G_n^\R(C_*)
  = \frac1{\|\mu_\R\|_1}.
  \]
  Here, $\|\cdot\|_1$ denotes the $\ell^1$-semi-norm on~$H_n(C_*;\R)$
  induced by the $\ell^1$-norm on the chain complex~$C_* \otimes_\Z \R$
  given by the marking of~$C_*$.
  Moreover, we use the convention~$1/0 = \infty$. 
\end{prop}
\begin{proof}
  Let $\omega \in H^n(C_*;\R)$ be the cohomology class dual
  to~$\mu_\R$.  By the duality principle~\cite{vbc} (which holds
  in this generality), we
  have~$\|\mu_\R\|_1 = 1/\|\omega\|_\infty$ (with~$1/\infty =
  0$). Therefore, it suffices to show that
  \[ G_n^\R(C_*) = \|\omega\|_\infty.
  \]

  We write~$E_n$ for the set of all chain extensions of~$C_*$
  over~$\R$ in degree~$n$.
  Then (when viewing~$\omega$ as set of all cocycles in~$\Hom_\Z(C_n,\R)$
  dual to~$\mu_\R$)
  \begin{align*}
    \Phi \colon 
    E_n & \longrightarrow \omega \\
    s & \longmapsto
    \bigl(c \mapsto \text{$\lambda$ with~$s(c) = \lambda \cdot \mu_\R$}\bigr)
    \\
    \Psi \colon
    \omega & \longrightarrow E_n \\
    f & \longmapsto \bigl(c \mapsto f(c) \cdot \mu_\R \bigr)
  \end{align*}
  are well-defined mutually inverse bijections. Indeed, for a chain
  extension~$s$, the chain extension condition ensures both that
  $\Phi(s)$ is a cocycle and that $\langle [\Phi(s)], \mu_\R \rangle =
  1$, and vice versa.

  If $B_0 := (\mu_\R)$ denotes the cointegral basis of~$H_n(C;\R)$
  given by~$\mu_\R$, then
  \[ \bigl| \Phi(s) \bigr|_\infty
  = \sup_{\sigma \in \Sigma_n} \bigl| \Phi(s)(\sigma)\bigr|
  = \sup_{\sigma \in \Sigma_n} \bigl|s(\sigma)\bigr|_1^{B_0}
  = \|s\|^{B_0}
  \]
  for all~$s \in E_n$; here, $\Sigma_n$ is the marked basis of~$C_n$.
  Hence,
  $\|\omega\|_\infty = \inf_{s \in E_n} \|s\|^{B_0}.
  $
  
  Finally, we show that other cointegral bases of~$H_n(C_*;\R)$ do
  not lead to better values: 
  Let $B \in B_n^\R(C_*)$, i.e., there is a~$b \in \Z \setminus \{0\}$
  with~$B = (1/b \cdot \mu_\R)$. Then, for all~$s \in E_n$, we obtain
  \[ \|s\|^B
  = |b| \cdot \|s\|^{B_0}
  \geq \|s\|^{B_0}.
  \]
  Taking the infimum over all~$B \in B_n^\R(C_*)$ shows that
  $\|\omega\|_\infty
  = G_n^\R(C_*)$.
\end{proof}

\begin{cor}\label{cor:gnarlsv}
  Let $M$ be an oriented closed connected $n$-manifold. Then
  \[ G_n^\R\bigl(C_*(M;\Z) \bigr)
  = \frac1 {\|M\|}.
  \]
  Here, $C_*(M;\Z)$ is the singular chain complex of~$M$, marked
  by the basis of singular simplices and $\|M\| := \|[M]_\R\|_1$
  denotes the simplicial volume of~$M$~\cite{vbc}.
\end{cor}
\begin{proof}
  We apply Proposition~\ref{prop:gnarlsingular} to the singular chain
  complex~$C_*(M;\Z)$ and the integral fundamental class~$[M]_\Z \in
  H_n(M;\Z)$ of~$M$.
\end{proof}

\begin{example}
  Thus, results on the spectrum of simplicial
  volume~\cite{heuer_loeh_spectrum, heuer_loeh_trans} show:
  \begin{itemize}
  \item Each rational number appears as gnarledness of
    singular chain complexes.
  \item Gnarledness of singular chain complexes can
    assume transcendental values.
  \end{itemize}
\end{example}
  
\begin{example}\label{exa:fincovsv}
  Let $d \in \N_{>0}$ and let $\overline M \to M$ be a $d$-sheeted
  covering of oriented closed connected $n$-manifolds. Then
  \begin{align*}
    2 \cdot L_n^\R\bigl(C_*(\overline M;\Z)\bigr)
    = 
    G_n^\R\bigl(C_*(\overline M;\Z)\bigr)
    = \frac1d \cdot G_n^\R\bigl(C_*(M;\Z)\bigr)
    = \frac2d \cdot L_n^\R\bigl(C_*(M;\Z)\bigr).
  \end{align*}
  Indeed, for the inner equality,
  we apply Corollary~\ref{cor:gnarlsv}, which
  relates gnarledness to the simplicial volume, and the
  fact that 
  simplicial volume is multiplicative under finite
  coverings~\cite{vbc}. For the relation with the cohomological
  lifting volume, we use that the top-homology of
  oriented closed connected manifolds is one-dimensional
  and Proposition~\ref{prop:liftgnarl}.
\end{example}

\begin{cor}
  Let $(S,r)$ be a one-relator presentation, where $r$ is a non-trivial
  element of the commutator subgroup of~$F(S)$ and let $\Gamma$ be
  the corresponding one-relator group.
  If $X$ is a model for the classifying space~$B\Gamma$, then
  \[ G_1^\R\bigl(C_*(X;\Z)\bigr) = \frac1{\|\Gamma\|}.
  \]
  Here, $C_*(X;\Z)$ is the singular chain complex of~$X$, marked by
  the basis of singular simplices and $\|\Gamma\|$ denotes the
  simplicial volume of~$\Gamma$~\cite{heuer_loeh_onerel}.
\end{cor}
\begin{proof}
  We apply Proposition~\ref{prop:gnarlsingular} to the singular chain
  complex~$C_*(X;\Z)$ and the integral fundamental class
  of~$\Gamma$ in~$H_1(X;\Z)$.
\end{proof}

%%%%%%%%%%%%%%%%%%%%%%%%%%%%%%%%%%%%%%%%%%%%%%%%%%%%%%%%%%%%%%%%%%%%%
\section{Irrelevance of coefficients}\label{sec:irrelevance}

We show that the cohomological lifting volume and gnarledness
of chain complexes are independent of the chosen coefficient
subfield of~$\R$.
In both cases, we use basic facts on approximation of reals
through numbers in a subfield and on (co)integrality properties.
For gnarledness, we in addition use base-change properties of size
invariants and universality.

%%%%%%%%%%%%%%%%%%%%%%%%%%%%%%
\subsection{Cohomological lifting volume}

\begin{prop}
  Let $\F$ be a subfield of~$\R$, let $k \in \N$, and
  let $C_*$ be a marked free $\Z$-chain complex with $\rk_\Z C_k <\infty$.
  Then
  \[ L_k^\Q(C_*) = L_k^\F(C_*).
  \]
\end{prop}
\begin{proof}
  This can be viewed as a version of invariance of filling volumes
  under inclusions of dense subcomplexes. We carry out the argument in
  this concrete situation.
  We consider the commutative diagram
  \[
  \begin{tikzcd}
    H^k(C_*;\Q)
    \ar[r]
    \ar[d,"\pi",description]
    & H^k(C_*;\F)
    \ar[d,"\pi_\F",description]
    \\
    H^k(C_*;\Q/\Z)
    \ar[r]
    & H^k(C_*;\F/\Z)
  \end{tikzcd}
  \]
  induced by the respective inclusions and projections on the coefficients.
  
  We have $L_k^\Q(C_*) \leq L_k^\F(C_*)$: 
  Let $\varphi \in H^k(C_*;\Q/\Z)$ be in the image of~$\pi$. It suffices to show that
  $L_k^\Q(C_*,\varphi) \leq L_k^\F(C_*)$.
  Let $\psi\in H^k(C_*;\F/\Z)$ be the image of~$\varphi$
  under coefficient change 
  and let $\widetilde \psi \in H^k(C_*;\F)$ be a $\pi_\F$-lift of~$\psi$.
  
  Let $g \in \Hom_\Z(C_k,\F)$ be a cocycle representing~$\widetilde \psi$. 
  Then, there exists a cochain~$h \in \Hom_\Z(C_{k-1},\F)$ such that
  $g + \delta h$ is rational and $[g + \delta h]$ is a $\pi$-lift of~$\varphi$. 
  Given~$\varepsilon \in \R_{>0}$, density of~$\Q$ in~$\F$ allows us to
  find a cochain~$h_\Q \in \Hom_\Z(C_{k-1},\Q)$ with~$|h_\Q - h| <\varepsilon$.
  We consider
  \[ f := g + \delta h - \delta h_\Q.
  \]
  By construction, $f$ is a rational cocycle with 
  $\pi\bigl([f]\bigr)
  = \varphi
  $
  and 
  \[ |f|_\infty \leq |g|_\infty + \|\delta^{k-1}\| \cdot |h - h_\Q|_\infty
  \leq |g|_\infty + \|\partial_k\| \cdot \varepsilon; 
  \]
  we know that $\|\partial_k\|$ is finite because $\rk_\Z C_k < \infty$. 
  Taking $\varepsilon \to 0$ and the infimum over all~$\widetilde \psi$
  (and~$g$, respectively) shows
  $L_k^\Q(C_*,\varphi) \leq L_k^\F(C_*,\psi) \leq L_k^\F(C_*)$. 
  
  Conversely, we have $L_k^\F(C_*) \leq L_k^\Q(C_*)$:
  Let $\psi \in H^k(C_*;\F/\Z)$ be in the image of~$\pi_\F$,
  say $\psi = \pi_\F(\widetilde \psi)$ with $\widetilde \psi
  \in H^k(C_*;\F)$.
  It suffices to show that $L_k^\F(C_*,\psi) \leq L_k^\Q(C_*)$.

  To this end, we first construct a class~$\widetilde \varphi$
  in~$H^k(C_*;\Q)$ ``sufficiently close to~$\widetilde \psi$''.
  Let $B$ be a $\Q$-basis of~$H_k(C_*;\Q)$
  induced by a $\Z$-basis of~$H_k(C_*)_\Q$.
  %and let $B_\F$
  %be the corresponding $\F$-basis of~$H_k(C_*;\F)$.
  Let $(\widetilde \varphi_\beta)_{\beta \in B}$ be the dual $\Q$-basis
  of~$H^k(C_*;\Q)$ and let $(\widetilde \psi_\beta)_{\beta \in B}$
  be the induced family in~$H^k(C_*;\F)$.
  Because $\rk_\Z C_k$ is finite,
  \[ K := \max_{\beta \in B} \|\widetilde \psi_\beta\|_\infty
  \]
  is finite. 
  By the universal coefficient theorem, $(\widetilde \psi_\beta)_{\beta \in B}$
  is an $\F$-basis of~$H^k(C_*;\F)$, namely the basis dual to the $\F$-basis
  of~$H_k(C_*;\F)$ induced by~$B$.
  We write $\widetilde \psi = \sum_{\beta \in B} \lambda_\beta \cdot \widetilde \psi_\beta$
  with~$\lambda_\beta \in \F$. 
  Given $\varepsilon \in \R_{>0}$, for all~$\beta \in B$,
  there exist~$\mu_\beta \in \Q$ with
  \[ |\mu_\beta - \lambda_\beta| < \varepsilon.
  \]
  We consider the class
  \[ \varphi
  := \pi \biggl(\sum_{\beta \in B} \mu_\beta \cdot \widetilde \varphi_\beta \biggr)
  \in H^k(C_*;\Q/\Z).
  \]

  Let $\widetilde \varphi \in H^k(C_*;\Q)$ be any $\pi$-lift of~$\varphi$
  and let $f \in \Hom_\Z(C_k,\Q)$ be a cocycle representing~$\widetilde \varphi$. 
  We set
  \[ \widetilde\varrho
  := [f]
  - \sum_{\beta \in B} (\mu_\beta - \lambda_\beta) \cdot \widetilde \psi_\beta
  \in H^k(C_*;\F).
  \]
  By construction, $\widetilde \varrho$ is a $\pi_\F$-lift of~$\psi$, because
  \begin{align*}
    \pi_\F(\widetilde \varrho)
    & = \pi_\F\bigl([f]\bigr)
    - \pi_\F\biggl(\sum_{\beta \in B} (\mu_\beta - \lambda_\beta) \cdot \widetilde \psi_\beta \biggr)
    \\
    & = \pi_\F\biggl(\sum_{\beta \in B} \mu_\beta \cdot \widetilde \psi_\beta \biggr)
    - \pi_\F\biggl(\sum_{\beta \in B} (\mu_\beta - \lambda_\beta) \cdot \widetilde \psi_\beta \biggr)
    = \pi_\F\biggl(\sum_{\beta \in B} \lambda_\beta \cdot \widetilde\psi_\beta\biggr)
    \\
    & = \pi_\F(\widetilde \psi) = \psi
    ,
  \end{align*}
  and we have  
  \begin{align*}
    \|\widetilde \varrho\|_\infty
    & \leq |f|_\infty
    + \sum_{\beta \in B} |\mu_\beta - \lambda_\beta| \cdot \|\widetilde \psi_\beta\|_\infty
    \leq |f|_\infty + \# B \cdot \varepsilon \cdot K.
  \end{align*}
  Taking $\varepsilon \to 0$ and the infimum over all~$\widetilde \varphi$
  (and $f$, respectively) shows $L_k^\F(C_*,\psi) \leq L_k^\Q(C_*,\varphi)
  \leq L_k^\Q(C_*)$.
\end{proof}

%%%%%%%%%%%%%%%%%%%%%%%%%%%%%%
\subsection{Size invariants}

Every size invariant over a subfield~$\F$ is also a size
invariant over~$\Q$. Conversely, every size invariant over~$\Q$
also satisfies monotonicity over~$\F$ on domain chain complexes
that satisfy a finiteness condition.

\begin{prop}[base change~I]\label{prop:sizebasechangeFQ}
  Let $\F$ be a subfield of~$\R$, let $k \in \N$,
  and let $I \colon \mCh \to [0,\infty]$ be a
  size invariant over~$\F$ in degree~$k$. Then
  $I$ is also a size invariant over~$\Q$ in degree~$k$.
\end{prop}
\begin{proof}
  We only need to check the monotonicity condition.
  If $f_* \colon C_* \otimes_\Z \Q \to D_*\otimes_\Z \Q$
  is a $\Q$-chain map inducing a $\Q$-isomorphism~$H_k(f_*)
  \colon H_k(C_*;\Q) \to H_k(D_*;\Q)$ satisfying~$H_k(f_*)
  (H_k(C_*)_\Q) \subset H_k(C_*)_\Q$, then (under
  the canonical natural isomorphism~$\args \otimes_\Z \F \cong
  \args \otimes_\Z \Q \otimes_\Q \F$) we may apply monotonicity
  over~$\F$ to 
  the base-changed
  chain map
  \[ f_*^\F := f_* \otimes_\Q \F\colon C_* \otimes_\Z \F
  \to D_* \otimes_\Z \F
  \]
  and obtain 
  $I(C_*)
  \leq \|f_k^\F\| \cdot I(D_*)
  \leq \|f_k\| \cdot I(D_*),
  $ 
  as desired.
\end{proof}

\begin{prop}[base change~II]\label{prop:sizebasechangeQF}
  Let $\F$ be a subfield of~$\R$, let $k \in \N$,
  and let $I \colon \mCh \to [0,\infty]$ be a
  size invariant over~$\Q$ in degree~$k$.
  Let $C_*$, $D_*$ in~$\mCh$ with~$\rk_\Z C_k < \infty$
  and let $f_* \colon C_* \otimes_\Z \F
  \to D_* \otimes_\Z \F$ be an $\F$-chain map that induces an 
  $\F$-iso\-morphism~$H_k(f_*) \colon H_k(C_*;\F) \to H_k(D_*;\F)$
  and that satisfies the integrality condition~$H_k(f_*) (H_k(C_*)_\F) \subset
  H_k(D_*)_\F$. Then: 
  \[ I(C_*) \leq \|f_k\| \cdot I(D_*).
  \]
\end{prop}
\begin{proof}
  Given~$\varepsilon \in \R_{>0}$, we apply the approximation lemma
  (Lemma~\ref{lem:chmapapprox}) below and obtain a chain
  map~$g_* \colon C_* \otimes_\Z \Q \to D_* \otimes_\Z \Q$ such
  that $H_k(g_*)$ is a $\Q$-isomorphism with~$H_k(g_*) (H_k(C_*)_\Q)
  \subset H_k(D_*)_\Q$ and~$\|g_k\| \leq \|f_k\| + \varepsilon$.
  Therefore, monotonicity of~$I$ over~$\Q$ gives us
  \[ I(C_*)
  \leq \|g_k\| \cdot I(D_*)
  \leq \bigl(\|f_k\| + \varepsilon\bigr) \cdot I(D_*).
  \]
  Taking~$\varepsilon \to 0$ completes the proof.
\end{proof}

We may view the canonical maps~$C_* \to C_*\otimes_\Z \Q \to C_*
\otimes_\Z \F$ as inclusions of subcomplexes. Elements in~$C_*
\otimes_\Z \F$ coming from~$C_*\otimes_\Z \Q$ will be called
\emph{rational}. Similarly, maps obtained via base change from~$\Q$
to~$\F$ or from~$\Z$ to~$\F$ will be called \emph{defined over~$\Q$}
or \emph{defined over~$\Z$}, respectively.

\begin{lemma}[rational approximation of chain maps]\label{lem:chmapapprox}
  Let $\F$ be a subfield of~$\R$, let $k \in \N$,
  let $C_*$, $D_*$ be in~$\mCh$ with~$\rk_\Z C_k < \infty$,
  and let $f_* \colon C_* \otimes_\Z \F \to D_* \otimes_\Z \F$
  be an $\F$-chain map with~$H_k(f_*) (H_k(C_*)_\F) \subset H_k(D_*)_\F$. 
  Then, for each~$\varepsilon \in \R_{>0}$,
  there exists an $\F$-linear chain
  map~$g_* \colon C_*\otimes_\Z \F \to D_*\otimes_\Z \F$
  that is defined over~$\Q$ with
  \[ H_k(g_*) = H_k(f_*) \colon H_k(C_*;\F) \to H_k(D_*;\F) 
  \qand
  \|g_k\| \leq \|f_k\| + \varepsilon.
  \]
\end{lemma}
\begin{proof}
  This is mainly a consequence of the fact that $C_*$ is defined
  over~$\Z$ and that $\Q$ is dense in~$\F$. 
  We construct~$g_*$ in degree~$k$, $k-1$ and lower degrees, $k+1$ and
  higher degrees, in this order.
  
  \emph{Degree~$k$:}
  Let $Z_k \subset C_k \otimes_\Z \F$ be the subspace
  of cycles. Because the boundary operator of~$C_*$ is defined over~$\Z$,
  we can choose an $\F$-basis~$B$ of~$Z_k$ consisting of rational chains.
  Let $p \colon C_k \otimes_\Z \F \to Z_k$ be an $\F$-linear projection that
  is defined over~$\Q$. We construct~$g_k$ separately on~$Z_k$
  and on~$\overline Z_k := \ker p$:
  \begin{itemize}
  \item On~$Z_k$: Let $c \in B$. By the integrality condition on~$H_k(f_*)$,
    the homology class~$[f_k(c)] \in H_k(D_*;\F)$ is rational. Thus, there exists
    a rational cycle~$z \in D_k \otimes_\Z \F$ and a chain~$b \in D_{k+1} \otimes_\Z \F$
    with
    \[ f_k(c) = z + \partial b. \]
    Because $\Q$ is dense in~$\F$, we may choose a rational chain~$b' \in D_{k+1} \otimes_\Z \F$
    with~$|\partial (b' - b)|_1 \leq \varepsilon$. We set
    $g_k(c) := z + \partial b'.
    $ 
    By construction, $g_k(c)$ is a rational cycle with 
    \[ \bigl[g_k(c)\bigr] = [z] = \bigl[f_k(c)\bigr]
    \qand
    \bigl| g_k(c) - f_k(c) \bigr|_1
    \leq \bigl| \partial (b' - b) \bigr|_1
    \leq \varepsilon 
    .
    \]
    We extend~$g_k$ linearly to~$Z_k$. Hence, $g_k|_{Z_k}$ is defined over~$\Q$.
  \item On~$\overline Z_k$:
    Let $\overline B$ be an $\F$-basis of~$\overline Z_k$ consisting
    of rational elements and let $c \in \overline B$. Because $\Q$
    is dense in~$\F$, there exists a rational chain~$g_k(c) \in D_k\otimes_\Z \F$
    with
    \[ \bigl| g_k(c) - f_k(c) \bigr|_1 \leq \varepsilon.
    \]
    We extend $g_k$ linearly to~$\overline Z_k$. Hence $g_k|_{\overline Z_k}$ is defined over~$\Q$.
  \end{itemize}
  As $C_k \otimes_\Z \F$ is the (internal) direct sum of the
  rationally defined subspaces~$Z_k$ and~$\overline Z_k$, we obtain a
  rationally defined $\F$-linear map~$g_k \colon C_k \otimes_\Z \F \to D_k
  \otimes_\Z \F$. Moreover, for all~$x \in C_k \otimes_\Z \F$, we have
  $p(x) \in Z_k$, $x-p(x) \in \overline Z_k$ and so 
  \begin{align*}
    \bigl| g_k(x) - f_k(x) \bigr|_1
    & 
    \leq \bigl| g_k(x- p(x)) - f_k(x-p(x))\bigr|_1
    + \bigl| g_k(p(x)) - f_k(p(x)) \bigr|_1
    \\
    &
    \leq \varepsilon \cdot \bigl| x- p(x)\bigr|_1^{\overline B}
    + \varepsilon \cdot \bigl| p(x) \bigr|_1^B
    \\
    &
    \leq \varepsilon \cdot \|\id - p \|^{\overline B} \cdot |x|_1
    + \varepsilon \cdot \|p\|^B \cdot |x|_1.
  \end{align*}
  We obtain~$\|g_k \| \leq \|f_k\| + K\cdot \varepsilon$, where
  $K := \|\id - p \|^{\overline B} + \|p\|^B$ 
  is finite because $\rk_\Z C_k < \infty$ and we take the operator
  norm with respect to~$|\cdot|_1$ on the domain. 
  
  \emph{Degree~$k-1$ and lower:}
  By construction of~$g_k$, we have $\ker \partial_k^C \subset
  \ker \partial_k^D \circ g_k$. Because the maps~$\partial_k^C$
  and $\partial_k^D \circ g_k$ are defined
  over~$\Q$, there exists a rationally defined $\F$-linear
  map~$g_{k-1} \colon C_{k-1} \otimes_\Z \F \to D_{k-1} \otimes_\Z \F$
  solving the mapping problem
  \[ g_{k-1} \circ \partial_k^C = \partial_k^D \circ g_k
  \qand
  \im g_{k-1} \subset \im \partial_k^D
  .
  \]
  In particular, $\partial_{k-1}^D \circ g_{k-1} = 0$ and we define
  $g_{k-2}, \dots, g_0$ to be the zero maps.
  
  \emph{Degree~$k+1$ and higher:}
  We construct an $\F$-linear map~$g_{k+1} \colon C_{k+1} \otimes_\Z \F
  \to D_{k+1} \otimes_\Z \F$ with~$\partial_{k+1} \circ g_{k+1} = g_k \circ \partial_{k+1}$
  and $g_{k+1} \circ \partial_{k+2} = 0$ 
  that is defined over~$\Q$; to this end, we argue as in the second part of
  the degree~$k$ case:
  We define $g_{k+1}$ as the zero map on~$\im \partial_{k+1}^C$.
  Because $\partial_{k+1}^C$ is defined over~$\Z$, the subspace~$\im \partial_{k+1}^C
  \subset C_{k+1} \otimes_\Z \F$
  is rationally defined. We choose a rationally defined complementary $\F$-subspace~$W
  \subset C_{k+1} \otimes_\Z \F$ and an $\F$-basis~$B$ of~$W$ consisting of
  rational chains. For~$c \in B$, the chain~$\partial c \in C_k \otimes_\Z\F$ is
  a cycle. By construction of~$g_k$, there exists a chain~$b \in D_{k+1} \otimes_\Z \F$
  with
  \[ g_k(\partial c) = f_k(\partial c) + \partial b
  = \partial \bigl( f_{k+1}(c) + b \bigr).
  \]
  Hence, the rational element~$g_k (\partial c)$ lies in the image
  of~$\partial_{k+1}^D$ on~$D_* \otimes_\Z \F$.
  Because the boundary operator of~$D_*$ is defined over~$\Z$, there
  also exists a rational solution~$d \in D_{k+1} \otimes_\Z \F$ with
  $g_k(\partial c) = \partial d.
  $ 
  We set~$g_{k+1}(c) := d$ and define~$g_{k+1}$ on all of~$W$
  through linear extension. By construction, $g_{k+1}$ is defined over~$\Q$ and
  $\partial \circ g_{k+1} = g_k \circ \partial$ as well as~$g_{k+1} \circ \partial_{k+2} = 0$.
  We define~$g_{k+2}, g_{k+3}, \dots$ to be the zero maps.
    
  By construction, $g_* \colon C_* \otimes_\Z \F \to D_* \otimes_\Z \F$
  is an $\F$-chain map that is defined over~$\Q$ and that satisfies
  $H_k(g_*) = H_k(f_*)$ as well as~$\|g_k\| \leq \|f_k\| + K \cdot \varepsilon$.
  Rescaling~$\varepsilon$ completes the proof. 
\end{proof}

%%%%%%%%%%%%%%%%%%%%%%%%%%%%%%
\subsection{Gnarledness}

The base change properties of size invariants
and universality show that gnarledness over~$\F$ and~$\Q$
coincide:

\begin{cor}\label{cor:gnarlcoeffirrelevant}
  Let $\F$ be a subfield of~$\R$ and let $k \in \N$.
  \begin{enumerate}
  \item If $C_*$ is in~$\mCh$, then~$G_k^\F(C_*) \leq G_k^\Q(C_*)$.
  \item If $C_*$ is in~$\mCh$ and~$\rk_\Z C_k < \infty$, then
    $G_k^\Q(C_*) \leq G_k^\F(C_*)$.
  \end{enumerate}
\end{cor}
\begin{proof}
  We know that $G_k^\F$ and $G_k^\Q$ are size invariants in degree~$k$
  over~$\F$ and~$\Q$, respectively
  (Proposition~\ref{prop:universality}).
  
  \emph{Ad~1.} By base change, $G_k^\F$ is a size invariant
  over~$\Q$ in degree~$k$ (Proposition~\ref{prop:sizebasechangeFQ}).
  Therefore, universality of~$G_k^\Q$ over~$\Q$ gives~$G_k^\F \leq G_k^\Q$
  (Proposition~\ref{prop:universality})

  \emph{Ad~2.} Let $C_*$ in~$\mCh$ and~$\rk_\Z C_k < \infty$.
  Then $G_k^\Q$ satisfies monotonicity over~$\F$ in degree~$k$
  for the domain complex~$C_*$ (Proposition~\ref{prop:sizebasechangeQF}). 
  The proof of universality (Proposition~\ref{prop:universality})
  hence shows that $G_k^\F$ dominates~$G_k^\Q$ on~$C_*$, i.e.,
  $G_k^\Q (C_*) \leq G_k^\F(C_*)$.
\end{proof}

Alternatively, one can also show that base change gives
a bijection between the cointegral bases of $H_k(\args;\Q)$
and~$H_k(\args;\F)$, that chain extensions over~$\Q$ lead
to chain extensions over~$\F$, and that chain extensions over~$\F$
can be approximated by chain extensions over~$\Q$.

%%%%%%%%%%%%%%%%%%%%%%%%%%%%%%%%%%%%%%%%%%%%%%%%%%%%%%%%%%%%%%%%%%%%%
\section{Controlling the norm of products}
\label{sec:normedkuenneth}

We use chain extensions to construct bounded chain homotopy
equivalences from a marked free $\R$-chain complex $C_*$ to its
homology complex $H_*(C_*)$ with trivial boundary operators.  In this
context, boundedness refers to the semi-norm on $H_*(C_*)$ that is
induced by the $\ell^1$-norm on $C_*$.

We use these chain homotopy equivalences to prove that the algebraic Künneth
isomorphism is bilipschitz under a finiteness condition
(Theorem~\ref{thm:opnorm-Kunneth}).

We apply this result to the singular chain complex and singular
homology with the canonical marked basis induced by the singular
simplices.  Under a finiteness condition, we then obtain 
norm-control for the topological Künneth isomorphism
(Corollary~\ref{cor:Kuennethtopo}) and a controlled behaviour of
product maps (Corollary~\ref{cor:productofbilip}).  These results
allow us to classify elements with trivial norm in the singular
homology of products of hyperbolic manifolds
(Example~\ref{exa:Kuennethhyperbolic}).

%%%%%%%%%%%%%%%%%%%%%%%%%%%%%%%%%%%%%%%%%%%%%%%%%%%%%%%%%%%%%%%%%%%%%
\subsection{Null spaces and reduced homology}
\label{subsec:nullspaces}

In order to understand the semi-norm on homology of a normed chain
complex, it is important to look at the ``defect" of the semi-norm,
i.e., the subspace of elements with trivial norm.

\begin{defi}[null space]
  If $\|\cdot \|$ is a semi-norm on an $\F$-vector space~$X$,
  we call $\Null(X):=\lbrace x\in X \mid \|x\| =0\rbrace $ the \emph{null space}
  of the semi-norm~$\|\cdot \|$.

  If we want to specify the semi-norm of the space~$X$ explicitly, we
  write~$\Null_{\|\cdot\|}(X)$ instead of~$\Null(X)$.
\end{defi}

We use the following notation in the context of normed chain complexes.

\begin{notation}
  \label{nota:homologynullspace}
  If $C_\ast$ is a normed $\F$-chain complex with induced semi-norm on
  $H_*(C_*)$ and $n\in\N$, we write
  \begin{align*}
    \Null_n(C_\ast)
    & := \Null(H_n(C_\ast)),
    \\
    \overline{H}_n(C_\ast)
    &:= H_n(C_\ast)/ \Null_n(C_\ast).
  \end{align*}
  The latter is called the \emph{reduced homology of $C_*$}.
\end{notation}

\begin{rem}
  The term \emph{reduced homology} of a chain complex $C_*$ of
  topological vector spaces usually refers to~$\ker(\partial_n)/
  \overline{\im(\partial_{n+1})}$.  There is a canonical isometric
  isomorphism between $\overline{H}_n(C_*)$ and $\ker(\partial_n)/
  \overline{\im(\partial_{n+1})}$.
\end{rem}

For better readability, (semi-)norms on a chain complex will be
written with single bars like $|\cdot |$ and the induced semi-norm on
homology will have double bars like $\|\cdot\|$.

We collect basic results about semi-norms and null-spaces that will be
needed in the upcoming sections.

\begin{lemma}
  \label{lemma:seminormbecomesnorm}
  Let $X$ be a semi-normed $\F$-vector space. Then the induced
  quotient semi-norm on~$X/ \Null(X)$ is a norm.  Furthermore, the
  canonical projection~$X\to X/\Null(X)$ is isometric with respect to
  the given semi-norm on~$X$ and its induced norm on~$X/\Null(X)$.
\end{lemma}

\begin{proof}
  This is true by definition of the induced quotient norm.
\end{proof}

\begin{example}[truncated $\ell^1$-semi-norm]
  Let $X$ be an $\F$-vector space and let $B$ be a basis of $X$. For a
  subfamily $B'$ of $B$ let
  \begin{align*}
    \|\cdot \|_1^{B,B'} \colon X &\to \R_{\geq 0} \\
    \sum_{b\in B} \lambda_b \cdot b &\mapsto \sum_{b\in B'}|\lambda_b|.
  \end{align*}
  Then, $\|\cdot \|_1^{B,B'}$ is a semi-norm on~$X$, called the
  \emph{truncated $\ell^1$-semi-norm on~$X$ associated to~$(B,B')$}.
  The null space of~$X$ with respect to~$\|\cdot \|_1^{B,B'}$ is the
  subspace of~$X$ spanned by $B\setminus B'$.
\end{example}

The equivalence of norms on finite dimensional $\R$-vector spaces
gives rise to an equivalence of semi-norms that have the same null
space.

\begin{lemma}
  \label{lemma:seminormequiv}
  Let $(X,\|\cdot \|)$ be a semi-normed $\R$-vector space such that
  $X/\Null(X)$ is finite dimensional. Then there is a basis~$B$ and a
  subfamily~$B'$ of~$B$ such that~$\|\cdot\|$ is equivalent to the
  truncated $\ell^1$-semi-norm on~$X$ associated to~$(B,B')$.
\end{lemma}

\begin{proof}
  We choose an $\R$-basis~$B_{\Null}$ of $\Null(X)$ and extend it to a
  basis~$B$ of~$X$. Then, for $B':=B\setminus B_{\Null}$, we have $N:=
  \Null(X)$ = $\Null_{\|\cdot \|_1^{B,B'}}(X)$. Hence, the semi-norms
  induced by~$\|\cdot \|$ and~$\|\cdot \|_1^{B,B'}$ on the quotient
  space~$X/N$ are indeed norms by
  Lemma~\ref{lemma:seminormbecomesnorm}. Since this quotient space is
  finite dimensional, these induced norms are equivalent. As the
  projection map is isometric for both semi-norms and their induced
  quotient norms, we deduce that the semi-norms on $X$ are also equivalent.
\end{proof}

%%%%%%%%%%%%%%%%%%%%%%%%%%%%%%%%%%%%%%%%%%%%%%%%%%%%%%%%%%%%%%%%%%%%%
\subsection{Boundedness of chain extensions with respect to the induced semi-norm on homology}
\label{subsec:snboundedchainext}
	
We give pragmatic conditions under which for a marked free $\Z$-chain
complex~$C_*$ there is a chain homotopy equivalence between
$C_\ast\otimes_\Z\R$ and $H_\ast (C_\ast;\R)$ that is bounded in every
degree.
% We will make use of Notation~\ref{nota:homologynullspace}. 
For a marked free chain complex $C_\ast$ we will, for every~$n\in\N$,
denote by $|\cdot|_1$ the $\ell^1$-norm on $C_n$ associated to the
marked basis. The induced semi-norm on homology will be denoted by
$\|\cdot\|_1$.

\begin{prop}
  \label{prop:bddchainext}
  Let $C_\ast$ be a marked free $\R$-chain complex. Let $n\in \N$ and
  assume $\overline{H}_n(C_*)$ is finite dimensional.  Let $B$ be a
  basis of $H_n(C_\ast)$ and $B'\subset B$ such that the projection
  map $\pi \colon \ker\partial_n \to H_n(C_\ast)$ is bounded with
  respect to $|\cdot|_1^C$ on $\ker\partial_n\subset C_n$ and $\|\cdot
  \|^{B,B'}_1$ on $H_n(C_\ast)$.  Then there is a chain extension of
  $C_\ast$ in degree $n$ that is bounded with respect to~$|\cdot|_1$
  and~$\|\cdot\|_1^{B,B'}$.
\end{prop}

\begin{proof}
  Since $\pi \colon \ker\partial_n \to H_n(C_\ast)$ is bounded with
  respect to $|\cdot|_1^C$ and $\|\cdot\|_1^{B,B'}$, the null space
  $N:= \Null_{\|\cdot \|_1^{B,B'}} (H_n(C_\ast))$ contains
  $\Null_n(C_\ast)$.  So there is a canonical projection
  \[ \overline{H}_n(C_*) = H_n(C_*)/\Null_n(C_\ast)\to H_n(C_*)/N.
  \]
  Since $\overline{H}_n(C_*)$ is finite dimensional by assumption,
  this implies that $H_n(C_*)/N$ is finite dimensional and so $B'$ is
  a finite subfamily of $B$.
    		
  In order to find an extension of $\pi$ to $C_n$, we use the
  Hahn--Banach theorem to extend the coordinate functions of $\pi$
  with respect to the basis $B$.  Let $b\in B$. We denote by $\ev_b$
  the $\R$-linear map given by
  \begin{align*}
    \ev_b \colon H_n(C_\ast) &\to \R\\
    B \ni b' &\mapsto  \delta_{b,b'}.
  \end{align*}
  We can describe $\pi$ via the coordinate maps $\ev_b$ since
  \begin{align*}
    \pi(x) = \sum_{b\in B} \ev_b(\pi(x))\cdot b
  \end{align*}
  for all $x\in \ker(\partial_n)$.
  
  For all $b \in B$, the map $\pi_b := \ev_b\circ \pi$ satisfies
  \[\|\pi_b\| = \|ev_b\circ \pi\|\leq \|ev_b\|\cdot \|\pi\| < \infty.
  \]
  For the calculation of the operator norms we consider the norm
  $|\cdot|_1^C$ on $\ker\partial_n$, the semi-norm
  $\|\cdot\|_1^{B,B'}$ on $H_n(C_*)$, and the absolute value on $\R$.
  
  For each of the finitely many elements $b\in B'$ we choose an
  extension~$s_b\colon C_n \to \R$ of~$\pi_b$ with~$\|s_b\| =
  \|\pi_b\|$. The existence of such extensions is guaranteed by the
  Hahn--Banach theorem.

  In order to extend the maps $\pi_b$ for the remaining basis elements
  $b\in B\setminus B'$, we choose a linear complement~$U$ of $\ker
  \partial_n$ in $C_n $ and define
  \begin{align*}
    s_b\colon C_n = \ker \partial_n \oplus U &\to \R\\
    (x,u)&\mapsto \pi_b(x).
  \end{align*}
  
  We will prove that the map
  \begin{align*}
    s\colon C_n &\to H_n(C_\ast) \\
    c &\mapsto \sum_{b\in B} s_b(c) \cdot b
  \end{align*}
  is a well-defined chain extension of $C_\ast$ in degree $n$:
  
  To prove that the map is well-defined, we need to check that, for
  every $c\in C_n$, its image is a finite sum.  This is true, as for
  $c\in \ker\partial_n$, we know that $s(c) = \pi (c)$ is a finite sum
  and if $c\in U$, then $s_b(c)$ is zero for all $b\in B\setminus B'$.
  As $B'$ is finite, $s(c)$ has only finitely
  many non-zero summands.  With $C_n = \ker\partial_n \oplus U$, we
  conclude that $s$ is well-defined.
		
  Since every coordinate map $s_b$ of $s$ is $\R$-linear, also the map
  $s$ is $\R$-linear and by construction it is a chain extension of
  $\pi$.
		
  Furthermore, $\|s\|$ is finite, as for every $c\in C_n$ we have
  \begin{align*}
    \bigl\|s(c)\bigr\| &= \biggl\| \sum_{b\in B} s_b(c) \cdot b\biggr\|_1^{B,B'} \leq \sum_{b\in B'} |s_b(c)| \\
    &\leq \sum_{b\in B'} \|s_b\| \cdot |c|_1 = \sum_{b\in B'} \|\pi_b\| \cdot |c|_1 \leq \# B' \cdot \|\pi\| \cdot |c|_1.
  \end{align*}
  Hence, $\|s\| \leq \# B'\cdot \|\pi\|$, which is finite, since $B'$
  is a finite family.
\end{proof}

\begin{cor}
  \label{cor:bddchainext}
  Let $C_\ast$ be a marked free $\Z$-chain complex and let $n\in \N$.
  If $\overline{H}_n(C_\ast;\R)$ is finite dimensional, then there is
  a chain extension of~$C_\ast$ in degree~$n$ over~$\R$ that is
  bounded with respect to the induced semi-norm on $H_n(C_\ast;\R)$.
\end{cor}

\begin{proof}
  Since $\overline{H}_n(C_\ast;\R)$ is finite dimensional, the induced
  semi-norm $\|\cdot \|_1^C$ on~$H_n(C_\ast;\R)$ is equivalent to a
  truncated $\ell^1$-semi-norm on $H_n(C_\ast;\R)$ by
  Lem\-ma~\ref{lemma:seminormequiv}.  Hence, the canonical projection
  $\pi \colon \ker\partial_n\otimes_\Z\R \to H_n(C_\ast;\R)$ that is
  bounded with respect to the canonical norm and semi-norm is also
  bounded with respect to this truncated $\ell^1$-semi-norm on
  homology.  We conclude using Proposition~\ref{prop:bddchainext} and
  using again the equivalence of the semi-norms on homology.
\end{proof}

\begin{cor}
  \label{cor:chhtpyequivtohomology}
  Let $C_\ast$ be a free marked $\Z$-chain complex such that
  $\overline{H}_n(C_\ast;\R)$ is finite dimensional in every degree
  $n\in\N$.  Then there is an $\R$-chain homotopy equivalence
  $(C_\ast,\partial_\ast) \to (H_\ast(C_*),0)$ that is bounded in
  every degree with respect to the induced semi-norm on homology.
\end{cor}

\begin{proof}
  By Corollary~\ref{cor:bddchainext}, we obtain for every $n\in\N$ a
  chain extension $s_n$ of~$C_*$ in degree $n$ over $\R$ that is
  bounded with respect to the induced semi-norm on~$H_n(C_*;\R)$.  The
  resulting chain map $s_*$ is a chain homotopy
  equivalence~\cite[Chapter~1, Theorem~8.4]{brown_cohomofgroups} as
  the $\R$-chain complex $C_*\otimes_\Z \R$ is free in every degree
  and bounded below.
\end{proof}

\begin{example}
  \label{exa:topspaces}
  Let $X$ be a topological space such that $\overline{H}_n(X;\R)$ is
  finite dimensional in every degree $n\in\N$.  Then, by
  Corollary~\ref{cor:chhtpyequivtohomology}, there is a chain homotopy
  equivalence $s_\ast$ between $C_\ast (X;\R)$ and $H_\ast(X; \R)$
  endowed with the trivial boundary operator that is bounded in every
  degree with respect to the $\ell^1$-norm on $C_\ast (X;\R)$
  associated to the basis of singular simplices and the induced
  semi-norm on $H_*(X;\R)$.
  
  Instances of topological spaces $X$ that fulfil these assumptions
  are closed manifolds and CW-complexes of finite type.
\end{example}

\subsection{The Künneth isomorphism is bilipschitz}
\label{subsec:normedkuenneth}
	
We consider the K\"unneth theorem, which describes the homology of a
tensor product complex by the homology of the factors.  The goal is to
study how this isomorphism behaves with respect to norms.  One key
ingredient in this comparison is the existence of a bounded chain
homotopy equivalence between an $\R$-chain complex and its homology
(Corollary~\ref{cor:bddchainext}).

\begin{thm}[K\"unneth isomorphism, normed version]
  \label{thm:opnorm-Kunneth}
  Let $C_\ast$ and $D_\ast$ be normed $\R$ chain complexes such that
  $\overline{H}_n(C_\ast;\R)$ is finite dimensional in every degree
  $n\in\N$.  Then, for each~$n \in \N$, the K\"unneth isomorphism
  \begin{align*}
    \varphi_n\colon \bigoplus_{p\in \N} H_p(C_\ast) \otimes_\R H_{n-p}(D_\ast) &\to H_n(C_\ast\otimes_\R D_\ast) \\
    [c]\otimes [d] &\mapsto [c\otimes d]
  \end{align*} 
  is bilipschitz with respect to $\normsum{\cdot}$ and $\normhom{\cdot}{C}$.
\end{thm}

In order to define $\normsum{\cdot}$ and $\normhom{\cdot}{C}$, we
introduce (semi-)norms on tensor products.

\begin{defi}[tensor (semi-)norm]
  Let $X$, $Y$ be normed (resp.\ semi-normed) $\R$-vector spaces. Then
  $X\otimes_\R Y$ carries a canonical norm (resp.\ semi-norm)
  $\|\cdot\|_\otimes$ that is given on $z\in X\otimes_\R Y$ by
  \begin{align*}
    \|z\|_\otimes := \inf \biggl\lbrace \sum_{i\in I} \|x_i\|_X\cdot \|y_i\|_Y \biggm|\; &I \text{ a finite set, } \fa{i\in I} x_i\in X \land y_i\in Y,  \\
    &z = \sum_{i\in I} x_i\otimes y_i \biggr\rbrace.
  \end{align*}
\end{defi}

This (semi-)norm on a tensor product is often referred to as
projective (tensor) norm to distinguish it from other tensor norms.
Since we only consider this specific tensor norm, we simply call it
tensor norm.

\begin{notation}
  Let $C_\ast$ and $D_\ast$ be semi-normed $\R$-chain complexes with
  norms \mbox{$|\cdot|_{C}$} and~$|\cdot|_{D}$, respectively, and let
  $n\in\N$.	
  \begin{itemize}
  \item We denote by $\normsum{\cdot}$ the semi-norm on
    $\bigoplus_{p\in \N} H_p(C_\ast) \otimes_\R H_{n-p}(D_\ast)$ given
    by the $\ell^1$-sum of the tensor semi-norms associated to $\|\cdot\|_C$ and
    $\|\cdot\|_D$ on each summand.
  \item We denote by $\normtensorcplx{\cdot}{C}$ the norm on
    \[(C_\ast\otimes_\R D_\ast)_n = \bigoplus_{p\in \N} C_p \otimes_\R D_{n-p}
    \]
    given by the $\ell^1$-sum of the tensor semi-norms associated to $|\cdot|_C$
    and $|\cdot|_D$ on each summand.  With $\normhom{\cdot}{C}$, we
    denote the induced semi-norm on the homology of the tensor product
    complex $C_\ast \otimes_\R D_\ast$.
  \item If $C_\ast$ and $D_\ast$ are marked free chain complexes, then
    $\normtensorcplx{\cdot}{C}$ is equivalent to the $\ell^1$-norm on
    $(C_\ast\otimes_\R D_\ast)_n$ given by the basis on the tensor
    product complex induced by the marked bases of the factors
    (Corollary~\ref{cor:tensornormis1norm}). We denote it
    by~$\normtensorcplx{\cdot}{C}_1$ and the induced semi-norm on
    homology by~$\normhom{\cdot}{C}_1$.
  \end{itemize}
\end{notation}

We will often use the following description of the tensor (semi-)norm:

\begin{lemma}
  \label{lemma:tensornormwithtruncted1norm}
  Let $X$ be an $\R$-vector space equipped with a truncated
  $\ell^1$-semi-norm associated to the basis $B$ and a subfamily
  $B'\subset B$ and let $Y$ be a (semi-) normed $\R$-vector space.
  Then, for all $z\in X\otimes_\R Y$, the norm $\|z\|_\otimes$ is
  attained by a specific representative, namely
  \[ \|z\|_\otimes = \sum_{b\in B} \|b\|_1^{B,B'} \cdot \|z_b\|_Y = \sum_{b\in B'} \|z_b\|_Y, 
  \]
  where $z= \sum_{b\in B} b\otimes z_b$ with $z_b\in Y$. Note that
  every element in $X\otimes_\R Y$ has such a representative and that
  it is unique.
\end{lemma}

\begin{proof}
  Let $z = \sum_{b\in B} b\otimes z_b\in X\otimes_\R Y$. By definition
  of the tensor semi-norm we have
  \[ \|z\|_\otimes \leq \sum_{b\in B} \|b\|_1^{B,B'} \cdot \|z_b\|_Y = \sum_{b\in B'} \|z_b\|_Y.
  \]
  Conversely, let $z= \sum_{i\in I} x_i \otimes y_i$ be a tensor
  representative of $z$ with $x_i =\sum_{b\in B} \lambda_{i,b}\cdot b$
  for all $i\in I$.  This means
  \[ z= \sum_{i\in I} x_i \otimes y_i = \sum_{i\in I} \sum_{b\in B} (\lambda_{i,b}\cdot b) \otimes y_i = \sum_{b\in B} b\otimes \Big(\sum_{i\in I}\lambda_{i,b}\cdot y_i\Big).
  \]
  Therefore we have $z_b = \sum_{i\in I}\lambda_{i,b}\cdot y_i$ for all $b\in B$.
  We compute
  \[ \sum_{i\in I} \|x_i\|_X \cdot \|y_i\|_Y = \sum_{i\in I} \sum_{b\in B'} |\lambda_{i,b} | \cdot \|y_i\|_Y \geq \sum_{b\in B'} \biggl\|\sum_{i\in I} \lambda_{i,b} \cdot y_i\biggr\|_Y = \sum_{b\in B'} \|z_b\|_Y.
  \]
  Taking the infimum over all tensor representatives of $z$ yields
  \[ \|z\|_\otimes \geq \sum_{b\in B'} \|z_b\|_Y = \sum_{b\in B} \|b\|_1^{B,B'} \cdot \|z_b\|_Y.\qedhere
  \]
\end{proof}

\begin{example}
  \label{cor:tensornormis1norm}
  Let $X$ and $Y$ be marked $\R$-vector spaces with the corresponding
  $\ell^1$-norms. Then, by Lemma~\ref{lemma:tensornormwithtruncted1norm}, the induced tensor norm on $X\otimes_\R Y$
  coincides with the $\ell^1$-norm on $X\otimes_\R Y$ associated to
  the basis that is given by the marked bases of $X$ and~$Y$.
\end{example}

\begin{proof}[Proof of Theorem~\ref{thm:opnorm-Kunneth}]
  We start by proving that the Künneth isomorphism~$\varphi_n$ in
  degree $n\in\N$ is bounded.  Since $\normsum{\cdot}$ is the
  $\ell^1$-sum of the tensor norms on the summands and only finitely
  many summands are non-zero, it suffices to show that $\varphi_n$ is
  bounded on every summand.  Let $p\in \lbrace 0,\ldots, n\rbrace$ and
  $x = \sum_{i\in I} \alpha_i\otimes \beta_i \in H_p(C_\ast)
  \otimes_\R H_{n-p}(D_\ast)$. For each $i\in I$, we choose a
  representing cycle $c_i$ of $\alpha_i$ and $d_i$ of $\beta_i$. Then
  \begin{align*}
    \bignormhom{\varphi_n(x)}{C} &= \normhom{\biggl[\sum_{i\in I}c_i\otimes d_i\biggr]}{C} \leq  \normtensorcplx{\sum_{i\in I}c_i\otimes d_i}{C} \\
     &\leq \sum_{i\in I}|c_i|_1^C\cdot |d_i|_1^D.
  \end{align*}
  Taking the infimum over all representing cycles and over all tensor
  representatives, we obtain
  \begin{align*}
    \bignormhom{\varphi_n(x)}{C}\leq \normsum{x}.
  \end{align*}

  In order to prove that the inverse of the Künneth isomorphism is
  bounded, we reduce to the case that $C_*$ has trivial boundary
  operator and give an explicit description of such an inverse map in
  this case.  Let $(H_\ast(C_\ast),0)$ be the homology chain complex
  of $C_\ast$ endowed with the semi-norm on homology of $C_\ast$.
  Then there is a chain homotopy equivalence $f_*: C_\ast\to
  H_\ast(C_\ast)$ that is bounded in every degree by
  Corollary~\ref{cor:chhtpyequivtohomology}.

  By assumption, in every degree $n$, the space
  $H_n(C_*)/\Null(H_n(C_*))$ is finite dimensional. So by
  Lemma~\ref{lemma:seminormequiv} we can choose a basis $B_n$ of
  $H_n(C_*)$ and a subfamily $B'_n$ of $B_n$ such that the
  semi-norm on $H_n(C_*)$ is equivalent to the truncated $\ell^1$-semi
  norm $\|\cdot\|_1^{B_n,B'_n}$.  We denote the homology complex
  endowed with this semi-norm by $A_*$ and denote the truncated
  $\ell^1$-semi-norm on~$A_*$ by $|\cdot|_A$. By construction, the identity map is
  is a chain homotopy equivalence between $A_*$ and $H_*(C_*)$ bounded in every degree 
  and the chain homotopy equivalence~$f_*\colon
  C_\ast\to A_*$ is also bounded in every degree.
  
  We have the following chain of isomorphisms that describes the
  inverse of the Künneth isomorphism for $C_*\otimes_\R D_*$:
  \[
   \begin{tikzcd}
     H_n(C_\ast\otimes_\R D_\ast) \ar[d,"H_n(f_*\otimes_\R \id_{D_*})\text{, bounded}", "\cong"']
     \\
     H_n(A_\ast \otimes_\R D_\ast) \ar[d,"\text{inverse of the Künneth isomorphism for } A_*\otimes_\R D_*", "\cong"']
     \\
    \bigoplus_{p\in\N} H_p(A_*) \otimes_\R H_{n-p}(D_\ast) \ar[d,"\text{bounded isomorphism by definition of } A_*", "\cong"']
     \\
     \bigoplus_{p\in\N} H_p(C_\ast) \otimes_\R H_{n-p}(D_\ast)
   \end{tikzcd}
  \]
   
  So in order to prove the theorem, it remains to show that the
  inverse of the Künneth isomorphism
  \begin{align*}
    \varphi_n\colon \bigoplus_{p\in \N} H_p(A_\ast) \otimes_\R H_{n-p}(D_\ast) &\to H_n(A_\ast\otimes_\R D_\ast) \\
    [a]\otimes [d] &\mapsto [a\otimes d]
  \end{align*} 
  is bounded.

  \begin{claim*}
    The map
    \begin{align*}
      \psi_n\colon H_n(A_\ast\otimes_\R D_\ast) &\to \bigoplus_{p\in \N} H_p(A_\ast) \otimes_\R H_{n-p}(D_\ast) \\
      \Bigg[\sum_{p=0}^n \sum_{b\in B_p} b\otimes y_{p,b} \Bigg] &\mapsto \sum_{p=0}^n \sum_{b\in B_p} [b]\otimes \big[y_{p,b}\big]
    \end{align*} 
    is well-defined, has bounded norm, and is inverse to the Künneth map.
  \end{claim*}

  \begin{proof}[Proof of Claim]
    In degree $n\in\N$, every $n$-chain of $A_*\otimes_\R D_*$ has a
    unique representative of the form $\sum_{p=0}^n \sum_{b\in B_p}
    b\otimes y_{p,b}$ with $y_{p,b}\in D_{n-p}$.  Furthermore, if
    $\sum_{p=0}^n \sum_{b\in B_p} b\otimes y_{p,b}$ is a cycle, then,
    for every $p\in\lbrace 1,\ldots, n\rbrace $, the chain $\sum_{b\in
      B_p} b\otimes \partial^D y_{p,b}=0$ is trivial, since
    $\partial^A$ is trivial.  By uniqueness of such a representative,
    we conclude that every $\partial^D y_{p,b}$ is zero and $y_{p,b}$
    is a cycle.  Hence, the expression on the right hand side of the
    definition of $\psi_n$ indeed describes classes in
    $\bigoplus_{p\in \N} H_p(A_\ast) \otimes_\R H_{n-p}(D_\ast)$.

    If two cycles $x = \sum_{p=0}^n \sum_{b\in B_p} b\otimes x_{p,b}$
    and $y=\sum_{p=0}^n \sum_{b\in B_p} b\otimes y_{p,b}$ represent
    the same class in $H_n(A_*\otimes_\R D_*)$, then we obtain a
    boundary
    \[
    x-y = \sum_{p=0}^n \sum_{b\in B_p} b\otimes (x_{p,b} - y_{p,b}).
    \] 
    Since the boundary operator of $A_*$ is zero, this means that for
    all $p\in\lbrace 0,\ldots, n\rbrace$ and $b\in B_p$ the element
    $x_{p,b} - y_{p,b}$ is a boundary in $D_{n-p}$.  Hence, we have
    $[x_{p,b}] = [y_{p,b}]$. All in all we conclude that $\psi_n$ is
    well-defined.
  
    One can easily see that the composition $\varphi_n \circ \psi_n$
    is the identity for every~$n\in\N$.  For proving that also $\psi_n
    \circ\varphi_n$ is identity, let $a\in A_p$ and $d\in D_{n-p}$ be
    cycles with $p\in \lbrace 0,\ldots, n\rbrace$. Let
    $(\lambda_b)_{b\in B_p}$ be a family in $\R$ such that $a=
    \sum_{b\in B_p} \lambda_b \cdot b$.  We compute
    \begin{align*}
      \psi_n\circ\varphi_n \big([a]\otimes [d]\big) &= \psi_n\big([a\otimes d]\big) = \psi_n\Biggl(\Biggl[\sum_{b\in B_p} b\otimes (\lambda_b \cdot d)\biggr]\Biggr) \\
      &= \sum_{b\in B_p} [b]\otimes [\lambda_b \cdot d] = \sum_{b\in B_p} [\lambda_b \cdot b]\otimes [d] = [a]\otimes [d].
    \end{align*}
    Hence, $\psi_n\circ \varphi_n$ is identity, as it is identity on a
    generating set. Put together, this means that $\psi_n$ is the
    inverse of $\varphi_n$.

    It remains to show that $\psi_n$ is bounded.  We observe that the
    quotient semi-norm $\|\cdot\|_A$ on $H_*(A_*)$ in degree $n\in\N$
    is the truncated $\ell^1$-semi-norm associated to the basis $[B_n]
    := ([b])_{b\in B_n}$ and the subfamily $[B'_n] := ([b'])_{b'\in
      B'_n}$ as the boundary operator of $A_*$ is zero.  By
    Lemma~\ref{lemma:tensornormwithtruncted1norm}, we can compute the
    semi-norm~$\normsum{\cdot}$ of an element in $H_p(A_\ast)
    \otimes_\R H_{n-p}(D_\ast)$ (for $p\in\lbrace 1,\ldots,n\rbrace$)
    by looking at its unique representative given by the bases $[B_p]$
    of $H_p(A_*)$, i.e., we have
    \[ \biggl\|{\sum_{b\in B_p} [b]\otimes [y_{p,b}]}\biggr\|_\otimes = \sum_{b\in B'_p} \bigl\|[y_{p,b}]\bigr\|_D.
    \]
    Analogously, let $\alpha \in H_n(A_\ast\otimes_\R D_\ast)$. Again
    by Lemma~\ref{lemma:tensornormwithtruncted1norm}, we observe that
    \[
    \text{\makebox[0pt]{%
        $\normhom{\alpha}{A} = \inf\Bigg\lbrace \sum_{b\in B'_p} |y_{p,b}|_D  \Biggm| \sum_{p=0}^n \sum_{b\in B_p} b\otimes y_{p,b} \text{ with } y_{p,b}\in D_{n-p} \text{ is a representative of }\alpha\Bigg \rbrace.$}}
    \]
    For such a representative $\sum_{p=0}^n \sum_{b\in B_p} b\otimes y_{p,b}$
    of $\alpha$ we compute
    \begin{align*}
        \normsum{\psi_n(\alpha)} = \biggl\|{\sum_{p=0}^n \sum_{b\in B_p} [b]\otimes [y_{p,b}]}\biggr\|_\otimes = \sum_{p=0}^n \sum_{b\in B'_p} \bigl\|[y_{p,b}]\bigr\|_D \leq \sum_{b\in B'_p} |y_{p,b}|_D.
    \end{align*}
    Taking the infimum over all such representatives of $\alpha$
    yields the boundedness of~$\psi_n$.
  \end{proof}
  We now proved that the Künneth isomorphism and its inverse are
  bounded. This means that the Künneth isomorphism is bilipschitz.
\end{proof}

\begin{cor}
  \label{cor:Kuennethtopo}
  Let $X$, $Y$ be topological spaces such that $X$ has finite
  dimensional reduced homology in each degree, e.g. $X$ is a closed
  manifold or a CW-complex of finite type. Then, for each $n\in\N$,
  the topological Künneth isomorphism
  \begin{align*}
    \bigoplus_{p\in \N} H_p(X;\R) \otimes_\R H_{n-p}(Y;\R) &\to H_n(X\times Y;\R) \\
    \alpha\otimes \beta &\mapsto \alpha\times \beta
  \end{align*}
  is bilipschitz.
\end{cor}

\begin{proof}
  As $X$ has finite dimensional reduced homology, the singular chain
  complexes $C_*(X;\R)$ and $C_*(Y;\R)$ together with the
  $\ell^1$-norms given by the bases of singular simplices fulfil the
  assumptions of Theorem~\ref{thm:opnorm-Kunneth}.  So the algebraic Künneth
  isomorphism for $C_*(X;\R)$ and $C_*(Y;\R)$ is bilipschitz and we
  reduce the claim to proving that, for every $n\in\N$, the
  isomorphism
  \begin{align*}
    H_n\bigl(C_*(X;\R)\otimes_\R C_*(Y;\R)\bigr) \to H_n(X\times Y;\R)
  \end{align*}
  induced by the Eilenberg--Zilber map is bilipschitz.
  Looking at explicit descriptions of the Eilenberg--Zilber map and
  its inverse~\cite[Exercises~VI.12.26]{dold_lectonalgtop}, we see
  that both are bounded on the level of chains.  This passes down to
  the level of homology with the induced semi-norm
  (Remark~\ref{rem:passboundedtohomology}).
\end{proof}

\begin{example}
  \label{exa:Kuennethhyperbolic}
  Let $M$ be an oriented, closed, connected hyperbolic manifold.  
  The surjectivity of the comparison map $H_b^n(M;\R)\to H^*(M;\R)$
  for $n=0$ and $n\geq 2$ shows that a class in the singular homology
  $H_n(M;\R)$ of $M$ has trivial norm if and only if the class is
  trivial or of degree $n=1$ \cite[8.3T]{gromov_hyperbolic}.
  
  Let $N$ be another oriented, closed, connected hyperbolic manifold.   
  By the norm control of the topological Künneth isomorphisms
  (Corollary~\ref{cor:Kuennethtopo}) we see that a class in
  $H_n(M\times N;\R)$ has trivial norm if and only if it is in the
  image of $H_1(M;\R)\otimes_\R H_{n-1}(N;\R) \oplus
  H_{n-1}(M;\R)\otimes_\R H_{1}(N;\R)$ under the Künneth isomorphism.

  Note that this can also be shown using duality results for the
  computation of the $\ell^1$-norm and the comparison maps for the
  product space $M\times N$.
\end{example}

The result about the norm control of the Künneth theorem can also be
used to control the norm of the product of bilipschitz chain maps.

\begin{cor}
  \label{cor:productofbilip}
  Let $X$, $X'$, $Y$, $Y'$ be topological spaces such that $X$ and
  $X'$ have finite dimensional reduced homology. 
  Let $f\colon X\to X'$ and $g\colon Y\to Y'$ be continuous maps. Let
  $n\in\N$ such that $H_p(f;\R)$ and $H_q(g;\R)$ are bilipschitz
  with respect to the $\ell^1$-semi-norm for all $p,q\in\lbrace
  0,\ldots, n\rbrace$. 
  Then the induced map $H_n(f\times g;\R) \colon H_n(X\times Y;\R) \to
  H_n(X'\times Y';\R)$ is also bilipschitz.
\end{cor}

\begin{proof}
  Since $X$ and $X'$ have finite dimensional reduced homology, we can
  apply that the topological Künneth isomorphism is bilipschitz,
  Corollary~\ref{cor:Kuennethtopo}, to reduce the problem to showing
  that the map
  \[
  \text{%
  \makebox[0pt]{%
  $\bigoplus_{p\in\N} H_p(f;\R) \otimes_\R H_{n-p}(g;\R)\colon \bigoplus_{p\in\N} H_p(X;\R) \otimes_\R H_{n-p}(Y;\R) \to \bigoplus_{p\in\N} H_p(X';\R) \otimes_\R H_{n-p}(Y';\R)$
  }}
  \]
  is bilipschitz. As there are only finitely many non-trivial
  summands, it suffices to prove that for every $p\leq n$ the map
  $H_p(f;\R) \otimes_\R H_{n-p}(g;\R)$ is bilipschitz. As $H_p(f;\R)$
  and $H_{n-p}(g;\R)$ are bilipschitz by assumption this follows from
  the following lemma.
\end{proof}

\begin{lemma}
  \label{lem:tensorbilip}
  Let $X$, $X'$, $Y$, $Y'$ be semi-normed $\R$-vector spaces such that
  $X'/\Null(X')$ is finite dimensional.
  Let $f\colon X\to X'$ and $g\colon Y\to Y'$ be bilipschitz linear maps. 
  Then also $f\otimes_\R g\colon X\otimes_\R Y\to X'\otimes_\R Y'$ is bilipschitz.
\end{lemma}

\begin{proof}
  Let $c_f$ and $c_g$ be bilipschitz constants for $f$ and $g$, respectively.
  Then, for $z\in X\otimes_\R Y$, we have $\|f\otimes_\R g (z)\|_\otimes \leq c_f\cdot c_g\|z\|_\otimes$:
  For every representative $\sum_{i\in I} x_i\otimes y_i$ of $z$, we compute
  \begin{align*}
   \bigl\|f\otimes_\R g (z) \bigr\|_\otimes &= \biggl\| \sum_{i\in I} f(x_i)\otimes g(y_i)\biggr\|_\otimes \leq \sum_{i\in I} \|f(x_i)\|_{X'}\cdot \|g(y_i)\|_{Y'} \\
   &\leq \sum_{i\in I} c_f\cdot c_g \cdot \|x_i\|_X\cdot \|y_i\|_Y. 
  \end{align*}
  Taking the infimum over all representatives of $z$ gives the desired inequality.

  For the lower bound, recall that the space $X'/\Null(X')$ is finite
  dimensional by assumption. Hence, all semi-norms on $X'$ respecting
  the null space $\Null(X')$ are equivalent by
  Lemma~\ref{lemma:seminormequiv}. 
  So without loss of generality we can assume that the semi-norm on
  $X'$ is a truncated $\ell^1$-semi-norm associated to a basis $B$ and
  a subfamily $B'\subset B$ compatible with $\im(f)$ and $\Null(X')$
  in the following way: 
  Let $B''$ be a basis of~$\Null(X')\cap \im(f)$. We extend $B''$ on
  the one hand to a basis $B_{\Null}$ of $\Null(X')$ and on the other
  hand to a basis $B_f$ of $\im(f)$. As $B_{\Null}\cap B_f = B''$ is a
  basis of $\Null(X')\cap \im(f)$, we know that $B_{\Null}\cup B_f$ is
  a basis of $\Null(X')\cup \im(f)$. 
  Hence, we can extend $B_{\Null}\cup B_f$ to a basis $B$ of $X'$ and
  define~$B':=B\setminus B_{\Null}$. 
  By construction, $B$ has both $B_{\Null}$ and $B_f$ as subfamily. 
  For $b\in B_f$ we fix a preimage $x_b\in X$. 
  For an element~$z\in X\otimes_\R Y$ there is a family $(y_b)_{b\in B_f}$ in $Y$ such that 
  \[  z = \sum_{b\in B_f} x_b\otimes y_b \mod \im\bigl(\ker(f)\otimes_\R Y\to X\otimes_\R Y\bigr).
  \]
  So we obtain
  \begin{align}
    \label{eq:reprofw}
    f\otimes_\R g(z) = \sum_{b\in B_f} f(x_b)\otimes g(y_b) = \sum_{b\in B_f} b\otimes g(y_b).
  \end{align} 
  Since elements in~$\ker(f)$ and hence elements in
  $\im\bigl(\ker(f)\otimes_R Y\to X\otimes_\R Y\bigr)$ have norm zero,
  we obtain
  \begin{align*}
    \|z\|_\otimes &\leq \sum_{b\in B_f} \|x_b\|_X\cdot \|y_b\|_Y \leq c_f\cdot c_g\cdot \sum_{b\in B_f} \|f(x_b)\|_{X'}\cdot \|g(y_b)\|_{Y'} \\
    &= c_f\cdot c_g \cdot \sum_{b\in B_f} \|b\|_{X'}\cdot \|g(y_b)\|_{Y'}.
  \end{align*}
  As we chose the norm on $X'$ to be the $\ell^1$-semi-norm associated
  to $B$ and $B'$ and as $B_f$ is a subfamily of $B$, we obtain
   \begin{align*}
    \|z\|_\otimes \leq c_f\cdot c_g \cdot \sum_{b\in B_f} \|b\|_X\cdot \|g(y_b)\|_Y =c_f\cdot c_g \cdot \|f(z)\|_\otimes
  \end{align*}
  using Lemma~\ref{lemma:tensornormwithtruncted1norm} and
  Equation~\ref{eq:reprofw}.  Thus, $f\otimes_\R g$ is bilipschitz.
\end{proof}

Replacing ``bilipschitz'' by ``isometric'' in the previous results
gives rise to the following question.

\begin{question}
  Let $X$, $X'$, $Y$, $Y'$ be topological spaces. 
  Let $f\colon X\to X'$ and $g\colon Y\to Y'$ be continuous maps. 
  Let $n\in\N$ such that $H_p(f;\R)$ and $H_q(g;\R)$ are isometric
  with respect to the $\ell^1$-semi-norm for all $p,q\in\lbrace
  0,\ldots, n\rbrace$.
  Does it follow that $H_n(f\times g;\R)\colon H_n(X\times Y;\R) \to
  H_n(X'\times Y';\R)$ is isometric with respect to the
  $\ell^1$-semi-norm?
\end{question}

In our proof of Corollary~\ref{cor:productofbilip} we use that the
Künneth isomorphism is bilipschitz
(Theorem~\ref{thm:opnorm-Kunneth}). Using our techniques, there is no
control over the bilipschitz constant of the Künneth isomorphism. 
So this approach is not suited to address the question above.

%%%%%%%%%%%%%%%%%%%%%%%%%%%%%%%%%%%%%%%%%%%%%%%%%%%%%%%
% A"I" declaration
\subsection*{Declaration}

We have not used any LLM~tools in the research for or writing  
of this article.

\enlargethispage{1cm}
%%%%%%%%%%%%%%%%%%%%%%%%%%%%%%%%%%%%%%%%%%%%%%%%%%%%%%%%%%%%%%%%%%%%%
%%%Bibliography:
{\small
\bibliographystyle{alpha}
\bibliography{bib}}

@article {vbc,
    AUTHOR = {Gromov, Michael},
     TITLE = {Volume and bounded cohomology},
   JOURNAL = {Inst. Hautes \'{E}tudes Sci. Publ. Math.},
  FJOURNAL = {Institut des Hautes \'{E}tudes Scientifiques. Publications
              Math\'{e}matiques},
    NUMBER = {56},
      YEAR = {1982},
     PAGES = {5--99 (1983)},
      ISSN = {0073-8301},
   MRCLASS = {53C20 (53C21 57R99 58E99)},
  MRNUMBER = {686042},
MRREVIEWER = {Karsten Grove},
       URL = {http://www.numdam.org/item?id=PMIHES_1982__56__5_0},
}

@article{CDMW,
  author={Chambers, Gregory R. and Dotterrer, Dominic and Manin, Fedor and Weinberger, Shmuel},
  title={Quantitative null-cobordism},
  note={With an appendix by F.~Manin and S.~Weinberger},
  journal={J.~Amer.\ Math.\ Soc.},
  volume={31},
  year={2018},
  number={4},
  pages={1165--1203},
}

@article {heuer_loeh_onerel,
    AUTHOR = {Heuer, Nicolaus and L{\"o}h, Clara},
     TITLE = {Simplicial volume of one-relator groups and stable commutator
              length},
   JOURNAL = {Algebr. Geom. Topol.},
  FJOURNAL = {Algebraic \& Geometric Topology},
    VOLUME = {22},
      YEAR = {2022},
    NUMBER = {4},
     PAGES = {1615--1661},
      ISSN = {1472-2747,1472-2739},
   MRCLASS = {20F65 (20E05 20F05 20F06 20J05 57M07)},
  MRNUMBER = {4495666},
MRREVIEWER = {Michael\ Hull},
       DOI = {10.2140/agt.2022.22.1615},
       URL = {https://doi.org/10.2140/agt.2022.22.1615},
}

@article {heuer_loeh_trans,
    AUTHOR = {Heuer, Nicolaus and L{\"o}h, Clara},
     TITLE = {Transcendental simplicial volumes},
   JOURNAL = {Ann.\ Inst.\ Fourier (Grenoble)},
  FJOURNAL = {Universit{\'e} de Grenoble. Annales de l'Institut Fourier},
    VOLUME = {74},
      YEAR = {2024},
    NUMBER = {2},
     PAGES = {763--781},
      ISSN = {0373-0956,1777-5310},
   MRCLASS = {57N65 (03D78 11J86 20J06 57M07)},
  MRNUMBER = {4748185},
       DOI = {10.5802/aif.3597},
       URL = {https://doi.org/10.5802/aif.3597},
}

@article {heuer_loeh_spectrum,
    AUTHOR = {Heuer, Nicolaus and L{\"o}h, Clara},
     TITLE = {The spectrum of simplicial volume},
   JOURNAL = {Invent. Math.},
  FJOURNAL = {Inventiones Mathematicae},
    VOLUME = {223},
      YEAR = {2021},
    NUMBER = {1},
     PAGES = {103--148},
      ISSN = {0020-9910,1432-1297},
   MRCLASS = {57N65 (20J05 57M07)},
  MRNUMBER = {4199441},
MRREVIEWER = {Andrea\ Tamburelli},
       DOI = {10.1007/s00222-020-00989-0},
       URL = {https://doi.org/10.1007/s00222-020-00989-0},
}

@book {brown_cohomofgroups,
	AUTHOR = {Brown, Kenneth S.},
	TITLE = {Cohomology of groups},
	SERIES = {Graduate Texts in Mathematics},
	VOLUME = {87},
	NOTE = {Corrected reprint of the 1982 original},
	PUBLISHER = {Springer-Verlag, New York},
	YEAR = {1994},
	PAGES = {x+306},
	ISBN = {0-387-90688-6},
	MRCLASS = {20J05 (20-02)},
	MRNUMBER = {1324339},
}

@article{bader_sauer,
  author = {Bader, Uri and Sauer, Roman},
  title  = {Uniform waist inequalities in codimension two for manifolds with
            {K}azhdan fundamental group},
  year   = {2026},
  volume = {36},
  pages = {981--997},
  journal = {Geom.\ Funct.\ Anal.},
}

@book {dold_lectonalgtop,
	AUTHOR = {Dold, Albrecht},
	TITLE = {Lectures on Algebraic Topology},
  SERIES = {Die Grundlehren der mathematischen Wissenschaften},
  VOLUME = {Band 200},
  PUBLISHER = {Springer-Verlag, New York-Berlin},
  YEAR = {1972},
  PAGES = {xi+377},
	ISBN = {0-387-05777-3},
	MRCLASS = {55-02},
  MRNUMBER = {415602},
}

@article {bucher,
    AUTHOR = {Bucher-Karlsson, Michelle},
     TITLE = {The simplicial volume of closed manifolds covered by {$\Bbb
              H^2\times\Bbb H^2$}},
   JOURNAL = {J. Topol.},
  FJOURNAL = {Journal of Topology},
    VOLUME = {1},
      YEAR = {2008},
    NUMBER = {3},
     PAGES = {584--602},
      ISSN = {1753-8416},
   MRCLASS = {53C20 (20J06 57M50)},
  MRNUMBER = {2417444},
MRREVIEWER = {Johannes F. Ebert},
       DOI = {10.1112/jtopol/jtn012},
       URL = {https://doi.org/10.1112/jtopol/jtn012},
}

@incollection {gromov_hyperbolic,
    AUTHOR = {Gromov, M.},
     TITLE = {Hyperbolic groups},
 BOOKTITLE = {Essays in group theory},
    SERIES = {Math. Sci. Res. Inst. Publ.},
    VOLUME = {8},
     PAGES = {75--263},
 PUBLISHER = {Springer, New York},
      YEAR = {1987},
      ISBN = {0-387-96618-8},
   MRCLASS = {20F32 (20F06 20F10 22E40 53C20 57R75 58F17)},
  MRNUMBER = {919829},
MRREVIEWER = {Christopher\ W.\ Stark},
       DOI = {10.1007/978-1-4613-9586-7\_3},
       URL = {https://doi.org/10.1007/978-1-4613-9586-7_3},
}

%%%%%%%%%%%%%%%%%%%%%%%%%%%%%%%%%%%%%%%%%%%%%%%%%%%%%%%%%%%%%%%%%%%%%
% author info
\vfill

\noindent
  \emph{Franziska Hofmann}\\[.5em]
  {\small
  \begin{tabular}{@{\qquad}l}
    Fakult\"at f\"ur Mathematik,
    Universit\"at Regensburg,
    93040 Regensburg\\
    %Germany\\
    \textsf{franziska2.hofmann@mathematik.uni-r.de}, 
    \url{https://homepages.uni-regensburg.de/~hof61178/}
  \end{tabular}}

\smallskip

\noindent
\emph{Clara L\"oh}\\[.5em]
  {\small
  \begin{tabular}{@{\qquad}l}
    Fakult\"at f\"ur Mathematik,
    Universit\"at Regensburg,
    93040 Regensburg\\
    %Germany\\
    \textsf{clara.loeh@mathematik.uni-r.de}, 
    \url{https://loeh.app.ur.de}
  \end{tabular}}
\end{document}